\documentclass[12pt]{article}
\usepackage{amsmath,amsthm,amscd}
\newtheorem{theorem}{Theorem}[section]
\newtheorem{corollary}[theorem]{Corollary}
\newtheorem{lemma}[theorem]{Lemma}
\newtheorem{proposition}[theorem]{Proposition}
\theoremstyle{remark}
\newtheorem{remark}[theorem]{Remark}
\numberwithin{equation}{section}
\DeclareMathOperator{\Gal}{Gal}
\DeclareMathOperator{\Fil}{Fil}
\DeclareMathOperator{\Ker}{Ker}

\DeclareMathOperator{\NS}{NS}
\DeclareMathOperator{\CH}{CH}
\DeclareMathOperator{\Hom}{Hom}
\DeclareMathOperator{\Sym}{Sym}
\DeclareMathOperator{\Sing}{Sing}
\title{The double cover of cubic surfaces\\
 branched along their Hessian}
\author{Atsushi Ikeda}
\date{}
\begin{document}
\maketitle
\footnote[0]{Partially supported by Grant-in-Aid for Young Scientists
(B) 20740014, Japan Society for the Promotion of Science.}
\begin{abstract}
 We prove the relation between the Hodge structure of the double cover
 of a nonsingular cubic surface branched along its Hessian and the Hodge
 structure of the triple cover of $\mathbf{P}^{3}$ branched along the
 cubic surface.
 And we introduce a method to study the infinitesimal variations of
 Hodge structure of the double cover of the cubic surface.
 Using these results, we compute the N\'{e}ron-Severi lattices for the
 double cover of a generic cubic surface and the Fermat cubic surface.
\end{abstract}
\section{Introduction}
Let $X\subset{\mathbf{P}^{3}}$ be a nonsingular cubic surface over the
complex numbers $\mathbf{C}$.
It is well-known that $X$ contains $27$ lines in $\mathbf{P}^{3}$.
A point $p\in{X}$ is called an Eckardt point if there are three lines
through $p$ on $X$.
The classification of nonsingular cubic surfaces by the configuration of
their Eckardt points is given in the book $\cite{s}$.
Although the configuration of the Eckardt points varies by a deformation
of $X$, the N\'{e}ron-Severi lattice for $X$ is constant.
In order to detect the difference of the configuration of the Eckardt
points, we consider the N\'{e}ron-Severi lattice for the double cover of
$X$ branched along its Hessian.
Let $B\subset{X}$ be the zeros of the Hessian of the defining equation
of $X$.
Then $B$ has at most node as its singularities, and a point $p\in{X}$ is
a node of $B$ if and only if $p$ is an Eckardt point on $X$.
Therefore an Eckardt point on $X$ corresponds to an ordinary double
point on the finite double cover $Y'$ over $X$ branched along $B$.
Let $\phi:Y\rightarrow{X}$ be the composition of the minimal resolution
of $Y'$ and the finite double cover.
Then an Eckardt point $e$ on $X$ corresponds to the $(-2)$-curve
$\phi^{-1}(e)$ on $Y$, and a line $L$ on $X$ splits by the pull-back
$\phi^{*}$ into two $(-3)$-curves $L^{+}$ and $L^{-}$ on $Y$, where we
can chose the component $L^{+}$ of $\phi^{*}L$ so that the union of $27$
rational curves $\bigcup_{L}L^{+}$ is a disjoint union.
We remark that $Y$ is a minimal surface of general type with the
geometric genus $4$, and the double cover $\phi$ is the canonical
morphism of $Y$.
In this paper, we regard an Eckardt point $e$ on $X$ as the class
$[\phi^{-1}(e)]$ in the N\'{e}ron-Severi lattice of $Y$, and we compute
the Hodge structure on $H^{2}(Y,\mathbf{Z})$.\par
There is another way to study cubic surfaces by using the Hodge
structure of some associated variety.
Let $\rho:V\rightarrow\mathbf{P}^{3}$ be the triple Galois cover
branched along a cubic surface $X$.
The Hodge structure on $H^{3}(V,\mathbf{Z})$ with the Galois action was
considered by Allcock, Carlson and Toledo $\cite{act}$ to understand the
moduli space of cubic surfaces as a ball quotient.
In this paper, we investigate the relation between the Hodge
structures $H^{2}(Y,\mathbf{Z})$ and $H^{3}(V,\mathbf{Z})$, and we prove
that there is an isomorphism
\begin{equation}\label{mi}
 \Bigl(\bigwedge^{2}H^{3}(V,\mathbf{Q})(1)
  \Bigr)^{\Gal{(\rho)}}\simeq
  \frac{H^{2}(Y,\mathbf{Q})}{\sum_{L}\mathbf{Q}L^{+}}
\end{equation}
of Hodge structures.
More precise statement in $\mathbf{Z}$-coefficients is given in
Theorem~$\ref{mt}$.
We remark that $V$ is a nonsingular cubic $3$-fold in $\mathbf{P}^{4}$,
and the Hodge structures of cubic $3$-folds were studied by
Clemens-Griffiths $\cite{cg}$ and Tjurin $\cite{t}$.
Let $S$ be the set of lines on a nonsingular cubic $3$-fold
$V\subset\mathbf{P}^{4}$.
It is a nonsingular projective surface, which is called the Fano surface
of lines on $V$.
Then the isomorphisms of Hodge structures
$H^{3}(V,\mathbf{Z})(1)\simeq
H^{1}(S,\mathbf{Z})$
and
$\bigwedge^{2}H^{1}(S,\mathbf{Q})\simeq
H^{2}(S,\mathbf{Q})$
are proved there.
In order to relate the Hodge structure $H^{2}(Y,\mathbf{Q})$ with
$H^{2}(S,\mathbf{Q})$, we regard the surface $Y$ as a kind of variety of
lines.
Let $\Lambda(\mathbf{P}^{3})$ be the Grassmannian variety of all lines
in $\mathbf{P}^{3}$.
We show that $Y$ is isomorphic to the variety
$$
Y_{3}=\{(p,L)\in\mathbf{P}^{3}\times\Lambda(\mathbf{P}^{3})\mid
\text{$L$ intersects $X$ at $p$ with the multiplicity $\geq{3}$}\},
$$
and the double cover $\phi:Y\rightarrow{X}$ corresponds to the first
projection
$Y_{3}\rightarrow{X};\
(p,L)\mapsto{p}$.
Then the second projection
$Y_{3}\rightarrow\Lambda(\mathbf{P}^{3});\
(p,L)\mapsto{L}$
is a birational morphism to its image
$Z_{3}\subset\Lambda(\mathbf{P}^{3})$,
and the Fano surface $S$ of the triple cover $V$ of $\mathbf{P}^{3}$ is
a triple cover of $Z_{3}$ by
$S\rightarrow{Z_{3}};\
L\mapsto{\rho(L)}$.
By the isomorphism
$
H^{2}(S,\mathbf{Q})^{\Gal{(\rho)}}\simeq
H^{2}(Z_{3},\mathbf{Q})
\simeq
\frac{H^{2}(Y,\mathbf{Q})}{\sum_{L}\mathbf{Q}L^{+}},
$
we get the isomorphism~$(\ref{mi})$.\par
By using this isomorphism (Theorem~$\ref{mt}$), we compute the
N\'{e}ron-Severi lattice $\NS{(Y)}$ of $Y$.
For a generic cubic surface $X$, we prove the theorem of
Noether-Lefschetz type (Theorem~\ref{nsg}), which says that
$\NS{(Y)}$ is generated by $(-3)$-curves on $Y$ corresponding to lines
on $X$ for a generic cubic surface.
We use the theory of the infinitesimal variations of Hodge structures
$\cite{cggh}$ to compute that the rank of $\NS{(Y)}$ is $28$ for a
generic cubic surface $X$.
We introduce a method to compute the Hodge cohomology
$H^{q}(Y,\Omega_{Y}^{p})$ for $Y$, which is a generalization of the
classical method by Griffiths $\cite{g}$.
And it enables us to compute the infinitesimal variations of Hodge
structure of $Y$.
In order to prove that the $(-3)$-curves on $Y$ generate the
N\'{e}ron-Severi group over $\mathbf{Z}$, we need the computation of the
determinant of the lattice, for which the identification in
Theorem~$\ref{mt}$ is used.
For a special cubic surface, the rank of $\NS{(Y)}$ is grater than $28$.
If $X$ is the Fermat cubic surface, then $\NS{(Y)}$ is of rank
$h^{1}(Y,\Omega_{Y}^{1})=44$, and the $\mathbf{Q}$-vector space
$\mathbf{Q}\otimes\NS{(Y)}$ is generated by $(-2)$-curves corresponding
to their Eckardt points and $(-3)$-curves corresponding to lines on $X$.
More precisely, the generator of $\NS{(Y)}$ over $\mathbf{Z}$ is given
in Theorem~$\ref{nsf}$.
For the proof of Theorem~$\ref{nsf}$, we use the computation of the
N\'{e}ron-Severi lattice of the Fano surface $S$ for the Fermat cubic
$3$-fold by Roulleau~$\cite{r}$.\par
The contents of this paper is the followings.
In Section~$\ref{vl}$, we introduce the variety $Y_{3}$ for a
nonsingular cubic surface $X$, and compute the numerical invariants for
the surface $Y_{3}$.
In Section~$\ref{dc}$, we prove that the first projection
$Y_{3}\rightarrow{X}$ is the double cover branched along the Hessian
$B$.
And we compute the intersection number on $Y=Y_{3}$ of the curve
$\phi^{-1}(e)$ corresponding to an Eckardt point $e$ on $X$ and the
curves $L^{\pm}$ corresponding to a line $L$ on $X$.
Then we give some relations of these curves in the N\'{e}ron-Severi
group of $Y$.
In Section~$\ref{ct}$, we review some results on nonsingular cubic
$3$-folds and their Fano surfaces in $\cite{cg}$ and $\cite{t}$.
In Section~$\ref{cs}$, we prove the relation $(\ref{mi})$ between the
Hodge structure of $Y$ and the Hodge structure of the triple cover
$V\rightarrow\mathbf{P}^{3}$.
And we determine the torsion part
$
\bigl(\frac{H^{2}(Y,\mathbf{Z})}{\sum_{L}\mathbf{Z}L^{+}}\bigr)_{\mathrm{tor}}
$
and the lattice structure on the free part
$
\bigl(\frac{H^{2}(Y,\mathbf{Z})}{\sum_{L}\mathbf{Z}L^{+}}\bigr)_{\mathrm{free}}.
$
In Section~$\ref{nsl}$, we compute the N\'{e}ron-Severi lattice of $Y$
for a generic cubic surface and the Fermat cubic surface.
In Section~$\ref{jr}$, we give a method to describe the Hodge cohomology
of $Y$, and we compute the infinitesimal variations of Hodge structure
for $Y$.
\section{Varieties of lines}\label{vl}
We denote by $\Lambda(\mathbf{P}^{n})$ the Grassmannian variety of all
lines in the projective space $\mathbf{P}^{n}$ over the complex numbers
$\mathbf{C}$, and by $\mathcal{O}_{\Lambda(\mathbf{P}^{n})}(1)$ the
line bundle which gives the Pl\"{u}cker embedding of
$\Lambda(\mathbf{P}^{n})$.
We denote by $\Gamma(\mathbf{P}^{n})$ be the flag variety of all pairs
$(p,L)$ of a point $p\in\mathbf{P}^{n}$ and a line
$L\subset\mathbf{P}^{n}$ which contains the point $p$;
$$
\Gamma(\mathbf{P}^{n})
=\{(p,L)\in\mathbf{P}^{n}\times\Lambda(\mathbf{P}^{n})\mid
p\in{L}\}.
$$
We remark that their canonical bundles are given by
$K_{\Lambda(\mathbf{P}^{n})}
\simeq\mathcal{O}_{\Lambda(\mathbf{P}^{n})}(-n-1)$
and
$K_{\Gamma(\mathbf{P}^{n})}
\simeq\varPhi^{*}\mathcal{O}_{\mathbf{P}^{n}}(-2)
\otimes\varPsi^{*}\mathcal{O}_{\Lambda(\mathbf{P}^{n})}(-n)$,
where $\varPhi:\Gamma(\mathbf{P}^{n})\rightarrow\mathbf{P}^{n}$ is the
first projection and
$\varPsi:\Gamma(\mathbf{P}^{n})\rightarrow\Lambda(\mathbf{P}^{n})$ is
the second projection.
Let
$
\mathcal{Q}_{\Lambda(\mathbf{P}^{n})}
=\{H^{0}(L,\mathcal{O}_{\mathbf{P}^{n}}(1)\vert_{L})\}
_{L\in\Lambda(\mathbf{P}^{n})}
$
be the tautological bundle on $\Lambda(\mathbf{P}^{n})$,
and let $\mathcal{S}$ be the subbundle of
$\varPsi^{*}\mathcal{Q}_{\Lambda(\mathbf{P}^{n})}$
whose fiber at $(p,L)\in\Gamma(\mathbf{P}^{n})$ is
$$
\mathcal{S}(p,L)
=\Ker{(H^{0}(L,\mathcal{O}_{\mathbf{P}^{n}}(1)\vert_{L})
{\longrightarrow}
H^{0}(p,\mathcal{O}_{\mathbf{P}^{n}}(1)\vert_{p})
)}.
$$
Then the Chow ring of $\Gamma(\mathbf{P}^{n})$ is
$$
\CH(\Gamma(\mathbf{P}^{n}))
\simeq
\mathbf{Z}[s,t]\big/\bigl(t^{n+1},\
\sum_{i=0}^{n}s^{n-i}t^{i}\bigr),
$$
where $s=c_{1}(\mathcal{S})$ and
$t=c_{1}(\varPhi^{*}\mathcal{O}_{\mathbf{P}^{n}}(1))$
(cf.\ \cite[(14.6)]{f}).\par
Let $X\subset\mathbf{P}^{3}$ be a nonsingular cubic surface.
We define subvarieties of $\Gamma(\mathbf{P}^{3})$ by
$$
Y_{m}=\{(p,L)\in\Gamma(\mathbf{P}^{3})\mid
\text{$L$ intersects $X$ at $p$ with the multiplicity $\geq{m}$}\}
$$
for $1\leq{m}\leq3$ and
$$
Y_{\infty}=\{(p,L)\in\Gamma(\mathbf{P}^{3})\mid
\text{$L$ is contained in $X$}\}.
$$
By the first projection $\varPhi$, $Y_{1}$ is a $\mathbf{P}^{2}$-bundle
over $X$, and $Y_{2}$ is a $\mathbf{P}^{1}$-bundle over $X$.
By \cite[Theorem~3.5]{i1}, $Y_{3}$ is a nonsingular projective
irreducible surface, and the first projection $\varPhi\vert_{Y_{3}}$ is a
generically finite morphism of degree $2$ over $X$.
Since $X$ contains $27$ lines in $\mathbf{P}^{3}$, $Y_{\infty}$ is a
disjoint union of $27$ rational curves.\par
Let $F\in{H^{0}(\mathbf{P}^{3},\mathcal{O}_{\mathbf{P}^{3}}(3))}$ be a
section which define the cubic surface $X$.
The restriction
$F\vert_{L}\in{H^{0}(L,\mathcal{O}_{\mathbf{P}^{3}}(3)\vert_{L})}$
is contained in the image of the natural injective homomorphism
$$
\mathcal{S}(p,L)^{\otimes{m}}\otimes
H^{0}(L,\mathcal{O}_{\mathbf{P}^{3}}(3-m)\vert_{L})
\longrightarrow
H^{0}(L,\mathcal{O}_{\mathbf{P}^{3}}(3)\vert_{L})
$$
if and only if the pair $(p,L)$ is contained in $Y_{m}$.
Hence, for $1\leq{m\leq3}$, the subvariety $Y_{m}$ is defined as the
zeros of a regular section of the vector bundle
$$
\frac{\varPsi^{*}\Sym^{3}\mathcal{Q}_{\Lambda(\mathbf{P}^{3})}}
{\mathcal{S}^{\otimes{m}}\otimes
\varPsi^{*}\Sym^{3-m}\mathcal{Q}_{\Lambda(\mathbf{P}^{3})}}
\simeq
\varPhi^{*}\mathcal{O}_{\mathbf{P}^{3}}(4-m)\otimes
\varPsi^{*}\Sym^{m-1}\mathcal{Q}_{\Lambda(\mathbf{P}^{3})}
$$
on $\Gamma(\mathbf{P}^{n})$, where the isomorphism is given in
\cite[\textsection{2}]{i2}.
\begin{proposition}\label{ni}
 $Y_{3}$ is a minimal surface of general type with the geometric
 genus $p_{g}(Y_{3})=4$, the irregularity $q(Y_{3})=0$ and the square of
 the canonical divisor $K_{Y_{3}}^{2}=6$, and the first projection
 $\varPhi\vert_{Y_{3}}$ is the canonical map of the surface $Y_{3}$.
\end{proposition}
\begin{proof}
 Since
 $$
 \begin{cases}
  \mathcal{O}_{\Gamma(\mathbf{P}^{3})}(Y_{1})\simeq
  \varPhi^{*}\mathcal{O}_{\mathbf{P}^{3}}(3),\\
  \mathcal{O}_{Y_{1}}(Y_{2})\simeq
  (\varPhi^{*}\mathcal{O}_{\mathbf{P}^{3}}(2)\otimes\mathcal{S})\vert_{Y_{1}}
  \simeq(\varPhi^{*}\mathcal{O}_{\mathbf{P}^{3}}(1)
  \otimes\varPsi^{*}\mathcal{O}_{\Lambda(\mathbf{P}^{3})}(1))\vert_{Y_{1}},\\
  \mathcal{O}_{Y_{2}}(Y_{3})\simeq
  (\varPhi^{*}\mathcal{O}_{\mathbf{P}^{3}}(1)\otimes\mathcal{S}^{\otimes{2}})
  \vert_{Y_{2}}
  \simeq(\varPhi^{*}\mathcal{O}_{\mathbf{P}^{3}}(-1))
  \otimes\varPsi^{*}\mathcal{O}_{\Lambda(\mathbf{P}^{3})}(2))\vert_{Y_{2}}
 \end{cases}
 $$
 and
 $K_{\Gamma(\mathbf{P}^{3})}
 =\varPhi^{*}\mathcal{O}_{\mathbf{P}^{3}}(-2)
 \otimes\varPsi^{*}\mathcal{O}_{\Lambda(\mathbf{P}^{3})}(-3)$,
 we have
 $$
 \begin{cases}
  K_{Y_{1}}\simeq(\varPhi^{*}\mathcal{O}_{\mathbf{P}^{3}}(1)
  \otimes\varPsi^{*}\mathcal{O}_{\Lambda(\mathbf{P}^{3})}(-3))\vert_{Y_{1}},\\
  K_{Y_{2}}\simeq(\varPhi^{*}\mathcal{O}_{\mathbf{P}^{3}}(2)
  \otimes\varPsi^{*}\mathcal{O}_{\Lambda(\mathbf{P}^{3})}(-2))\vert_{Y_{2}},\\
  K_{Y_{3}}\simeq(\varPhi^{*}\mathcal{O}_{\mathbf{P}^{3}}(1))\vert_{Y_{3}}.
 \end{cases}
 $$
 Since
 $H^{i}(\Gamma(\mathbf{P}^{3}),\varPhi^{*}\mathcal{O}_{\mathbf{P}^{3}}(-3))
 =0$
 and
 $H^{i}(\Gamma(\mathbf{P}^{3}),\varPhi^{*}\mathcal{O}_{\mathbf{P}^{3}}(-2))
 =0$
 for any $i$, the restriction induces isomorphisms
 $$
 \begin{cases}
  H^{i}(\Gamma(\mathbf{P}^{3}),\mathcal{O}_{\Gamma(\mathbf{P}^{3})})
  \simeq
  H^{i}(Y_{1},\mathcal{O}_{Y_{1}}),\\
  H^{i}(\Gamma(\mathbf{P}^{3}),\varPhi^{*}\mathcal{O}_{\mathbf{P}^{3}}(1))
  \simeq
  H^{i}(Y_{1},(\varPhi^{*}\mathcal{O}_{\mathbf{P}^{3}}(1))\vert_{Y_{1}})
 \end{cases}
 $$
 for any $i$.
 Since
 $H^{i}(\Gamma(\mathbf{P}^{3}),\varPhi^{*}\mathcal{O}_{\mathbf{P}^{3}}(j)
 \otimes\varPsi^{*}\mathcal{O}_{\Lambda(\mathbf{P}^{3})}(-1))=0$
 for any $i$ and $j$, we have
 $H^{i}(Y_{1},(\varPhi^{*}\mathcal{O}_{\mathbf{P}^{3}}(j)
 \otimes\varPsi^{*}\mathcal{O}_{\Lambda(\mathbf{P}^{3})}(-1))
 \vert_{Y_{1}})=0$ for any $i$ and $j$, hence the restriction induces
 isomorphisms
 $$
 \begin{cases}
  H^{i}(Y_{1},\mathcal{O}_{Y_{1}})
  \simeq
  H^{i}(Y_{2},\mathcal{O}_{Y_{2}}),\\
  H^{i}(Y_{1},(\varPhi^{*}\mathcal{O}_{\mathbf{P}^{3}}(1))\vert_{Y_{1}})
  \simeq
  H^{i}(Y_{2},(\varPhi^{*}\mathcal{O}_{\mathbf{P}^{3}}(1))\vert_{Y_{2}})
 \end{cases}
 $$
 for any $i$, and the dimension of these cohomology groups are
 $$
 h^{i}(Y_{2},\mathcal{O}_{Y_{2}})=
 h^{i}(\Gamma(\mathbf{P}^{3}),\mathcal{O}_{\Gamma(\mathbf{P}^{3})})=
 \begin{cases}
  1&\text{if $i=0$},\\
  0&\text{if $i\neq0$},
 \end{cases}
 $$
 and
 $$
 h^{i}(Y_{2},(\varPhi^{*}\mathcal{O}_{\mathbf{P}^{3}}(1))\vert_{Y_{2}})=
 h^{i}(\Gamma(\mathbf{P}^{3}),\varPhi^{*}\mathcal{O}_{\mathbf{P}^{3}}(1))=
 \begin{cases}
  4&\text{if $i=0$},\\
  0&\text{if $i\neq0$}.
 \end{cases}
 $$
 By the exact sequence
 $$
 0\longrightarrow
 K_{Y_{2}}
 \longrightarrow
 (\varPhi^{*}\mathcal{O}_{\mathbf{P}^{3}}(1))\vert_{Y_{2}}
 \longrightarrow
 K_{Y_{3}}
 \longrightarrow0
 $$
 and the duality
 $$
 H^{i}(Y_{2},K_{Y_{2}})
 \simeq
 H^{3-i}(Y_{2},\mathcal{O}_{Y_{2}})^{\vee},
 $$
 we have $p_{g}(Y_{3})=4$ and $q(Y_{3})=0$, and $\varPhi\vert_{Y_{3}}$
 is the canonical map.
 Since
 $K_{Y_{3}}\simeq(\varPhi^{*}\mathcal{O}_{\mathbf{P}^{3}}(1))\vert_{Y_{3}}$
 is nef and the image of the canonical map is the surface $X$,
 the surface $Y_{3}$ is a minimal surface of general type.\par
 Since $Y_{3}$ is defined as the zeros of a regular section of the
 vector bundle
 $$
 \frac{\varPsi^{*}\Sym^{3}\mathcal{Q}_{\Lambda(\mathbf{P}^{3})}}
 {\mathcal{S}^{\otimes{3}}}
 \simeq
 \varPhi^{*}\mathcal{O}_{\mathbf{P}^{3}}(1)\otimes
 \varPsi^{*}\Sym^{2}\mathcal{Q}_{\Lambda(\mathbf{P}^{3})},
 $$
 its class in the Chow ring of $\Gamma(\mathbf{P}^{3})$ is
 $$
 [Y_{3}]
 =c_{3}(\varPhi^{*}\mathcal{O}_{\mathbf{P}^{3}}(1)\otimes
 \varPsi^{*}\Sym^{2}\mathcal{Q}_{\Lambda(\mathbf{P}^{3})})
 =6s^{2}t+15st^{2}+6t^{3}
 \in\CH^{3}(\Gamma(\mathbf{P}^{3})),
 $$
 hence
 $$
 K_{Y_{3}}^{2}
 =\deg{(c_{1}(\varPhi^{*}\mathcal{O}_{\mathbf{P}^{3}}(1))^{2}\cdot[Y_{3}])}
 =6.
 $$
\end{proof}
\begin{remark}
 Proposition~\ref{ni} implies that the Hodge number
 $h^{1}(Y_{3},\Omega_{Y_{3}}^{1})=44$.
 Minimal surfaces with such numerical invariants are classified by
 Horikawa, and $Y_{3}$ is of type Ib in \cite{h}.
 Since $Y_{3}$ is simply connected by \cite[Theorem~12.1]{h}, we have
 $H_{1}(Y_{3},\mathbf{Z})=0$, hence $H^{i}(Y_{3},\mathbf{Z})$ has no
 torsion element for any $i$.
\end{remark}
Since the cubic surface $X$ is recovered as the image of the canonical
map of $Y_{3}$, we have the following Torelli type theorem.
\begin{corollary}
 The isomorphism class of the cubic surface $X$ is uniquely determined by
 the isomorphism class of $Y_{3}$.
\end{corollary}
\begin{proposition}\label{inf}
 Each component of $Y_{\infty}$ is a $(-3)$-curve on $Y_{3}$.
\end{proposition}
\begin{proof}
 Since
 $\mathcal{O}_{Y_{3}}(Y_{\infty})\simeq
 \mathcal{S}^{\otimes{3}}\vert_{Y_{3}}$,
 the self intersection number of $Y_{\infty}$ on $Y_{3}$ is
 $$
 (Y_{\infty}.Y_{\infty})
 =\deg{(c_{1}(\mathcal{S}^{\otimes{3}})^{2}\cdot[Y_{3}])}
 =-81.
 $$
 The self intersection number of a component of $Y_{\infty}$ is less
 than $-1$ because $K_{Y_{3}}$ is nef, and the component is not a
 $(-2)$-curve because its image by the canonical map is a line in
 $\mathbf{P}^{3}$.
 Since $Y_{\infty}$ is a disjoint union of $27$ rational curves,
 each component of $Y_{\infty}$ is $(-3)$-curve on $Y_{\infty}$.
\end{proof}
\begin{remark}\label{z3}
 The second projection
 $$
 \varPsi\vert_{Y_{3}}:
 Y_{3}\longrightarrow\Lambda(\mathbf{P}^{3});\
 (p,L)\longmapsto{L},
 $$
 is birational to its image $Z_{3}=\varPsi(Y_{3})$,
 which induces an isomorphism
 $Y_{3}\setminus{Y_{\infty}}\simeq
 Z_{3}\setminus{Z_{\infty}}$,
 where
 $Z_{\infty}=\{L\in{\Lambda(\mathbf{P}^{3})\mid{L\subset{X}}}\}$
 is equal to the singular locus of $Z_{3}$.
\end{remark}
\section{The double cover branched along Hessian}\label{dc}
For simplicity, we denote the first projection
$\varPhi\vert_{Y_{3}}:Y_{3}\rightarrow{X}$ by $\phi:Y\rightarrow{X}$.
Let $R$ be the ramification divisor of $\phi:Y\rightarrow{X}$.
Since $R$ is the zeros of the determinant of the differential
$d\phi:T_{Y}\rightarrow\phi^{*}{T_{X}}$,
its class in $\CH^{1}(Y)$ is
$$
[R]=c_{1}(K_{Y}\otimes\phi^{*}{K_{X}}^{\vee})
=c_{1}((\varPhi^{*}\mathcal{O}_{\mathbf{P}^{3}}(2))\vert_{Y}).
$$
We denote by $B=\phi_{*}R$ the branch divisor of $\phi$.
Let
$
F(x_{0},\dots,x_{3})
\in\mathbf{C}[x_{0},\dots,x_{3}]
$
be a cubic polynomial which defines the nonsingular cubic surface $X$.
\begin{proposition}
 $B\subset{X}$ is the zeros of the Hessian
 $$
 \det{
 \Bigl(\frac{\partial^{2}{F}}{\partial{x_{i}}\partial{x_{j}}}\Bigr)
 _{0\leq{i,j}\leq{3}}}
 \in{H^{0}(X,\mathcal{O}_{\mathbf{P}^{3}}(4)\vert_{X})}.
 $$
\end{proposition}
\begin{proof}
 For $p=[a_{0}:a_{1}:a_{2}:a_{3}]\in\mathbf{P}^{3}$, if
 $a_{0}\neq0$, then there is an isomorphism
 $$
 \mathbf{P}^{2}\overset{\sim}{\longrightarrow}
 \varPhi^{-1}(p)\subset\Gamma(\mathbf{P}^{3});\
 q=[b_{1}:b_{2}:b_{3}]\longmapsto{(p,L_{(p,q)})},
 $$
 where $L_{(p,q)}$ denotes the line through the
 points $p$ and $[0:b_{1}:b_{2}:b_{3}]$
 in $\mathbf{P}^{3}$;
 $$
 L_{(p,q)}=\{[a_{0}t_{0}:a_{1}t_{0}+b_{1}t_{1}:
 \dots:a_{3}t_{0}+b_{3}t_{1}]
 \in\mathbf{P}^{3}\mid[t_{0}:t_{1}]\in\mathbf{P}^{1}\}.
 $$
 For $0\leq{i}\leq{3}$, we set a polynomial $F_{i}(x,z)$ on variables
 $(x_{0},\dots,x_{3},z_{1},\dots,z_{3})$ inductively by
 $$
 F_{0}(x,z)=F(x_{1},\dots,x_{3})
 $$
 and
 \begin{equation}\label{defeq}
 F_{i}(x,z)=\frac{1}{i}\sum_{j=1}^{3}
 \frac{\partial{F_{i-1}}}{\partial{x_{j}}}(x,z)z_{j}.
 \end{equation}
 Since
 \begin{multline*}
  F(a_{0}t_{0},a_{1}t_{0}+b_{1}t_{1},a_{2}t_{0}+b_{2}t_{1},
  a_{3}t_{0}+b_{3}t_{1})\\
  =F_{0}(a,b)t_{0}^{3}+F_{1}(a,b)t_{0}^{2}t_{1}
  +F_{2}(a,b)t_{0}t_{1}^{2}+F_{3}(a,b)t_{1}^{3},
 \end{multline*}
 if $p\in{X}$, then
 $$
 \phi^{-1}(p)
 \simeq
 \{q=[b_{1}:b_{2}:b_{3}]\in\mathbf{P}^{2}\mid
 F_{1}(a,b)=0,\ F_{2}(a,b)=0
 \}.
 $$
 $p\in{X}$ is contained in $B$ if and only if
 there exists
 $[b_{1}:b_{2}:b_{3}]\in\mathbf{P}^{2}$ such that
 $F_{1}(a,b)=F_{2}(a,b)=0$ and the rank of the matrix
 \begin{multline*}
 \begin{pmatrix}
  \frac{\partial{F_{1}}}{\partial{z_{1}}}(a,b)&
  \frac{\partial{F_{1}}}{\partial{z_{2}}}(a,b)&
  \frac{\partial{F_{1}}}{\partial{z_{3}}}(a,b)\\
  \frac{\partial{F_{2}}}{\partial{z_{1}}}(a,b)&
  \frac{\partial{F_{2}}}{\partial{z_{2}}}(a,b)&
  \frac{\partial{F_{2}}}{\partial{z_{3}}}(a,b)
 \end{pmatrix}\\
 =
 \begin{pmatrix}
  \frac{\partial{F}}{\partial{x_{1}}}(a)&
  \frac{\partial{F}}{\partial{x_{2}}}(a)&
  \frac{\partial{F}}{\partial{x_{3}}}(a)\\
  \sum_{j=1}^{3}\frac{\partial^{2}{F}}{\partial{x_{j}}\partial{x_{1}}}(a)
  b_{j}&
  \sum_{j=1}^{3}\frac{\partial^{2}{F}}{\partial{x_{j}}\partial{x_{2}}}(a)
  b_{j}&
  \sum_{j=1}^{3}\frac{\partial^{2}{F}}{\partial{x_{j}}\partial{x_{3}}}(a)
  b_{j}
 \end{pmatrix}
 \end{multline*}
 is less than $2$.
 Since
 $(\frac{\partial{F}}{\partial{x_{1}}}(a),
 \frac{\partial{F}}{\partial{x_{2}}}(a),
 \frac{\partial{F}}{\partial{x_{3}}}(a))
 \neq(0,0,0)$,
 the condition on the rank of the matrix is equivalent to the existence
 of $b_{0}\in\mathbf{C}$ such that
 $$
 b_{0}
 \begin{pmatrix}
  \frac{\partial{F}}{\partial{x_{1}}}(a)&
  \frac{\partial{F}}{\partial{x_{2}}}(a)&
  \frac{\partial{F}}{\partial{x_{3}}}(a)
 \end{pmatrix}
 +
 \begin{pmatrix}
  b_{1}&b_{2}&b_{3}
 \end{pmatrix}
 \begin{pmatrix}
  \frac{\partial^{2}{F}}{\partial{x_{1}}^{2}}(a)&
  \frac{\partial^{2}{F}}{\partial{x_{1}}\partial{x_{2}}}(a)&
  \frac{\partial^{2}{F}}{\partial{x_{1}}\partial{x_{3}}}(a)\\
  \frac{\partial^{2}{F}}{\partial{x_{2}}\partial{x_{1}}}(a)&
  \frac{\partial^{2}{F}}{\partial{x_{2}}^{2}}(a)&
  \frac{\partial^{2}{F}}{\partial{x_{2}}\partial{x_{3}}}(a)\\
  \frac{\partial^{2}{F}}{\partial{x_{3}}\partial{x_{1}}}(a)&
  \frac{\partial^{2}{F}}{\partial{x_{3}}\partial{x_{2}}}(a)&
  \frac{\partial^{2}{F}}{\partial{x_{3}}^{2}}(a)
 \end{pmatrix}
 =0.
 $$
 Then $F_{1}(a,b)=0$ implies $F_{2}(a,b)=0$, because
 \begin{align*}
 F_{2}(a,b)&=\frac{1}{2}
 \begin{pmatrix}
 b_{1}&b_{2}&b_{3}
 \end{pmatrix}
 \begin{pmatrix}
 \frac{\partial^{2}{F}}{\partial{x_{1}}^{2}}(a)&
 \frac{\partial^{2}{F}}{\partial{x_{1}}\partial{x_{2}}}(a)&
 \frac{\partial^{2}{F}}{\partial{x_{1}}\partial{x_{3}}}(a)\\
 \frac{\partial^{2}{F}}{\partial{x_{2}}\partial{x_{1}}}(a)&
 \frac{\partial^{2}{F}}{\partial{x_{2}}^{2}}(a)&
 \frac{\partial^{2}{F}}{\partial{x_{2}}\partial{x_{3}}}(a)\\
 \frac{\partial^{2}{F}}{\partial{x_{3}}\partial{x_{1}}}(a)&
 \frac{\partial^{2}{F}}{\partial{x_{3}}\partial{x_{2}}}(a)&
 \frac{\partial^{2}{F}}{\partial{x_{3}}^{2}}(a)
 \end{pmatrix}
 \begin{pmatrix}
 b_{1}\\
 b_{2}\\
 b_{3}
 \end{pmatrix}\\
 &=-\frac{b_{0}}{2}
 \begin{pmatrix}
 \frac{\partial{F}}{\partial{x_{1}}}(a)&
 \frac{\partial{F}}{\partial{x_{2}}}(a)&
 \frac{\partial{F}}{\partial{x_{3}}}(a)
 \end{pmatrix}
 \begin{pmatrix}
 b_{1}\\
 b_{2}\\
 b_{3}
 \end{pmatrix}
 =-\frac{b_{0}}{2}F_{1}(a,b).
 \end{align*}
 Hence, $p\in{X}$ is contained in $B$ if and only if
 there exists
 $[b_{0}:b_{1}:b_{2}:b_{3}]\in\mathbf{P}^{3}$ such that
 $$
 \begin{pmatrix}
  b_{0}&b_{1}&b_{2}&b_{3}
 \end{pmatrix}
 \begin{pmatrix}
  0&
  \frac{\partial{F}}{\partial{x_{1}}}(a)&
  \frac{\partial{F}}{\partial{x_{2}}}(a)&
  \frac{\partial{F}}{\partial{x_{3}}}(a)\\
  \frac{\partial{F}}{\partial{x_{1}}}(a)&
  \frac{\partial^{2}{F}}{\partial{x_{1}}^{2}}(a)&
  \frac{\partial^{2}{F}}{\partial{x_{1}}\partial{x_{2}}}(a)&
  \frac{\partial^{2}{F}}{\partial{x_{1}}\partial{x_{3}}}(a)\\
  \frac{\partial{F}}{\partial{x_{2}}}(a)&
  \frac{\partial^{2}{F}}{\partial{x_{2}}\partial{x_{1}}}(a)&
  \frac{\partial^{2}{F}}{\partial{x_{2}}^{2}}(a)&
  \frac{\partial^{2}{F}}{\partial{x_{2}}\partial{x_{3}}}(a)\\
  \frac{\partial{F}}{\partial{x_{3}}}(a)&
  \frac{\partial^{2}{F}}{\partial{x_{3}}\partial{x_{1}}}(a)&
  \frac{\partial^{2}{F}}{\partial{x_{3}}\partial{x_{2}}}(a)&
  \frac{\partial^{2}{F}}{\partial{x_{3}}^{2}}(a)
 \end{pmatrix}
 =
 \begin{pmatrix}
  0&
  0&
  0&
  0
 \end{pmatrix},
 $$
 and it is equivalent to
 \begin{align*}
  0&=
  \det{
  \begin{pmatrix}
   0&
   \frac{\partial{F}}{\partial{x_{1}}}(a)&
   \frac{\partial{F}}{\partial{x_{2}}}(a)&
   \frac{\partial{F}}{\partial{x_{3}}}(a)\\
   \frac{\partial{F}}{\partial{x_{1}}}(a)&
   \frac{\partial^{2}{F}}{\partial{x_{1}}^{2}}(a)&
   \frac{\partial^{2}{F}}{\partial{x_{1}}\partial{x_{2}}}(a)&
   \frac{\partial^{2}{F}}{\partial{x_{1}}\partial{x_{3}}}(a)\\
   \frac{\partial{F}}{\partial{x_{2}}}(a)&
   \frac{\partial^{2}{F}}{\partial{x_{2}}\partial{x_{1}}}(a)&
   \frac{\partial^{2}{F}}{\partial{x_{2}}^{2}}(a)&
   \frac{\partial^{2}{F}}{\partial{x_{2}}\partial{x_{3}}}(a)\\
   \frac{\partial{F}}{\partial{x_{3}}}(a)&
   \frac{\partial^{2}{F}}{\partial{x_{3}}\partial{x_{1}}}(a)&
   \frac{\partial^{2}{F}}{\partial{x_{3}}\partial{x_{2}}}(a)&
   \frac{\partial^{2}{F}}{\partial{x_{3}}^{2}}(a)
  \end{pmatrix}}\\
  &=
  \frac{a_{0}^{2}}{4}\cdot
  \det{
  \begin{pmatrix}
   \frac{\partial^{2}{F}}{\partial{x_{0}}^{2}}(a)&
   \frac{\partial^{2}{F}}{\partial{x_{0}}\partial{x_{1}}}(a)&
   \frac{\partial^{2}{F}}{\partial{x_{0}}\partial{x_{2}}}(a)&
   \frac{\partial^{2}{F}}{\partial{x_{0}}\partial{x_{3}}}(a)\\
   \frac{\partial^{2}{F}}{\partial{x_{1}}\partial{x_{0}}}(a)&
   \frac{\partial^{2}{F}}{\partial{x_{1}}^{2}}(a)&
   \frac{\partial^{2}{F}}{\partial{x_{1}}\partial{x_{2}}}(a)&
   \frac{\partial^{2}{F}}{\partial{x_{1}}\partial{x_{3}}}(a)\\
   \frac{\partial^{2}{F}}{\partial{x_{2}}\partial{x_{0}}}(a)&
   \frac{\partial^{2}{F}}{\partial{x_{2}}\partial{x_{1}}}(a)&
   \frac{\partial^{2}{F}}{\partial{x_{2}}^{2}}(a)&
   \frac{\partial^{2}{F}}{\partial{x_{2}}\partial{x_{3}}}(a)\\
   \frac{\partial^{2}{F}}{\partial{x_{3}}\partial{x_{0}}}(a)&
   \frac{\partial^{2}{F}}{\partial{x_{3}}\partial{x_{1}}}(a)&
   \frac{\partial^{2}{F}}{\partial{x_{3}}\partial{x_{2}}}(a)&
   \frac{\partial^{2}{F}}{\partial{x_{3}}^{2}}(a)
  \end{pmatrix}}.
 \end{align*}
 Hence $B$ is defined by the Hessian on
 $X\setminus\{x_{0}\neq0\}$.
 In the same way, we can show that $B$ is defined by the Hessian
 on $X\setminus\{x_{i}\neq0\}$ for $1\leq{i}\leq{3}$.
\end{proof}
Let $E$ be the sum of all components of $R$ which contract to
points by $\phi$, and let $D$ be the divisor such that $R=D+E$.
For a line $L$ on $X$, we denote by $L^{+}$ the corresponding component
of $Y_{\infty}$;
$$
L^{+}=\{(p,L')\in{\Gamma(\mathbf{P}^{3})}\mid{L'=L}\}.
$$
Let $L^{-}$ be the other component of $\phi^{*}(L)$ dominating $L$ by
$\phi$, and let $Y_{\infty}^{-}$ be the sum of $L^{-}$ for all lines on
$X$.
A point $p$ on the cubic surface $X$ is called an Eckardt point if there
are three lines through $p$ on $X$.
\begin{theorem}\label{div}
 The divisor $D$ is a disjoint union of nonsingular curves,
 $E$ is a disjoint union of $(-2)$-curves on $Y$, and
 $Y_{\infty}^{-}$ is a disjoint union of $(-3)$-curves on $Y$.
 The divisors $R+Y_{\infty}$,
 $R+Y_{\infty}^{-}$ and $E+Y_{\infty}+Y_{\infty}^{-}$ are reduced
 simple normal crossing divisors.
 The branch divisor $B$ has at most nodes as its singularities, and the
 singular locus of  $B$ is equal to the set of Eckardt points of $X$.
 A line $L$ on $X$ intersects $B$ at two points with each multiplicity
 $2$, and
 $$
 \phi^{*}L=L^{+}+L^{-}
 +\sum_{e\in{L\cap\Sing{(B)}}}\phi^{-1}(e).
 $$
\end{theorem}
First, we normalize for $p\in{X}$ the equation of $X$ by a
transformation of the homogeneous coordinate in order to introduce a
local coordinate of $X$ around $p$ and to compute the local equation of
these divisors.
\begin{lemma}\label{nor}
 Let
 $F(x)=\sum{c_{ijk}x_{0}^{3-i-j-k}}x_{1}^{i}x_{2}^{j}x_{3}^{k}$
 be an equation of a nonsingular cubic surface $X$, and let $p$ be a
 point on $X$.
 \begin{enumerate}
  \item If $\phi^{-1}(p)$ is a set of distinct two points, then
	$F(x)$ is normalized by a transformation of the homogeneous
	coordinate to satisfy
	$p=[1:0:0:0]$,
	$c_{000}=c_{100}=c_{010}=c_{200}=c_{020}=0$
	and $c_{001}=c_{110}=1$.
  \item If $\phi^{-1}(p)$ is a point, then
	$F(x)$ is normalized by a transformation of the homogeneous
	coordinate to satisfy
	$p=[1:0:0:0]$,
	$c_{000}=c_{100}=c_{010}=c_{200}=c_{110}=0$
	and $c_{001}=c_{020}=c_{210}=1$.
  \item If $\phi^{-1}(p)\simeq{\mathbf{P}^{1}}$, then
	$F(x)$ is normalized by a transformation of the homogeneous
	coordinate to satisfy
	$p=[1:0:0:0]$,
	$c_{000}=c_{100}=c_{010}=c_{200}=c_{110}=c_{020}=c_{210}=c_{120}=0$
	and $c_{001}=3c_{300}=3c_{030}=1$.
 \end{enumerate}
\end{lemma}
\begin{proof}
 First, we can chose a homogeneous coordinate $[x_{0}:\dots:x_{3}]$ as
 $p=[1:0:0:0]$.
 Then $p\in{X}$ implies that $c_{000}=0$.
 Since $X$ is nonsingular at $p$,
 $(c_{100},c_{010},c_{001})\neq(0,0,0)$.
 We may assume that $c_{001}\neq0$.
 By the the transformation
 $$
 \begin{pmatrix}
  x_{1}\\
  x_{2}\\
  x_{3}
 \end{pmatrix}
 \longmapsto
 \begin{pmatrix}
  1&0&0\\
  0&1&0\\
  c_{100}&c_{010}&c_{001}
 \end{pmatrix}
 \begin{pmatrix}
  x_{1}\\
  x_{2}\\
  x_{3}
 \end{pmatrix},
 $$
 we may assume that $(c_{100},c_{010},c_{001})=(0,0,1)$.
 \begin{enumerate}
  \item We consider the case where $\phi^{-1}(p)$ is a set of
	distinct two points.
	Then the quadratic form
	$c_{200}x_{1}^{2}+c_{110}x_{1}x_{2}+c_{020}x_{2}^{2}$
	is factorized into independent linear forms;
	$$
	c_{200}x_{1}^{2}+c_{110}x_{1}x_{2}+c_{020}x_{2}^{2}
	=(\alpha_{1}{x_{1}}+\alpha_{2}{x_{2}})
	(\beta_{1}{x_{1}}+\beta_{2}{x_{2}}).
	$$
	By the transformation
	$$
	\begin{pmatrix}
	 x_{1}\\
	 x_{2}
	\end{pmatrix}
	\longmapsto
	\begin{pmatrix}
	 \alpha_{1}&\alpha_{2}\\
	 \beta_{1}&\beta_{2}
	\end{pmatrix}
	\begin{pmatrix}
	 x_{1}\\
	 x_{2}
	\end{pmatrix},
	$$
	$F(x)$ is normalized to satisfy
	$(c_{200},c_{110},c_{020})=(0,1,0)$.
  \item We consider the case where $\phi^{-1}(p)$ is a point.
	Then the quadratic form
	$c_{200}x_{1}^{2}+c_{110}x_{1}x_{2}+c_{020}x_{2}^{2}$
	is the square of a nonzero linear form;
	$$
	c_{200}x_{1}^{2}+c_{110}x_{1}x_{2}+c_{020}x_{2}^{2}
	=(\alpha_{1}{x_{1}}+\alpha_{2}{x_{2}})^{2},
	$$
	and we may assume that $\alpha_{2}\neq0$.
	By the transformation
	$$
	\begin{pmatrix}
	 x_{1}\\
	 x_{2}
	\end{pmatrix}
	\longmapsto
	\begin{pmatrix}
	 1&0\\
	 \alpha_{1}&\alpha_{2}
	\end{pmatrix}
	\begin{pmatrix}
	 x_{1}\\
	 x_{2}
	\end{pmatrix},
	$$
	we may assume that
	$(c_{200},c_{110},c_{020})=(0,0,1)$.
	If $c_{210}\neq0$, then by the transformation
	$$
	\begin{pmatrix}
	 x_{1}\\
	 x_{2}
	\end{pmatrix}
	\longmapsto
	\begin{pmatrix}
	 \sqrt{c_{210}}&0\\
	 0&1
	\end{pmatrix}
	\begin{pmatrix}
	 x_{1}\\
	 x_{2}
	\end{pmatrix},
	$$
	$F(x)$ is normalized to satisfy $c_{210}=1$.
	If $c_{210}=0$ and $c_{300}\neq0$, then by the transformation
	$$
	\begin{pmatrix}
	 x_{1}\\
	 x_{2}
	\end{pmatrix}
	\longmapsto
	\begin{pmatrix}
	 1&-\frac{1}{3c_{300}}\\
	 0&1
	\end{pmatrix}
	\begin{pmatrix}
	 x_{1}\\
	 x_{2}
	\end{pmatrix},
	$$
	$F(x)$ is normalized to satisfy $c_{210}=1$.
	If $(c_{300},c_{210})=(0,0)$, then $X$ is singular at
	$[a:1:0:0]$, where $a$ is a root of the quadratic equation
	$$
	\frac{\partial{F}}{\partial{x_{3}}}(s,1,0,0)
	=s^{2}+c_{101}s+c_{201}=0.
	$$
  \item We consider the case where
	$\phi^{-1}(p)\simeq{\mathbf{P}^{1}}$.
	Then we have $(c_{200},c_{110},c_{020})=(0,0,0)$, and the cubic
	form
	$c_{300}x_{1}^{3}+c_{210}x_{1}^{2}x_{2}
	+c_{120}x_{1}x_{2}^{2}+c_{030}x_{2}^{3}$
	is factorized into nonzero linear forms;
	$$
	c_{300}x_{1}^{3}+c_{210}x_{1}^{2}x_{2}
	+c_{120}x_{1}x_{2}^{2}+c_{030}x_{2}^{3}
	=(\alpha_{1}{x_{1}}+{\alpha_{2}}{x_{2}})
	(\beta_{1}{x_{1}}+\beta_{2}{x_{2}})
	(\gamma_{1}{x_{1}}+\gamma_{2}{x_{2}}).
	$$
	We have $\alpha_{1}\beta_{2}-\alpha_{2}\beta_{1}\neq0$,
	$\beta_{1}\gamma_{2}-\beta_{2}\gamma_{1}\neq0$,
	and $\gamma_{1}\alpha_{2}-\gamma_{2}\alpha_{1}\neq0$, because
	for example
	if $\alpha_{1}\beta_{2}-\alpha_{2}\beta_{1}=0$, then $X$ is
	singular at $[a:-\alpha_{2}:\alpha_{1}:0]$, where $a$ is a root
	of the quadratic equation
	$$
	\frac{\partial{F}}{\partial{x_{3}}}
	(s,-\alpha_{2},\alpha_{1},0)
	=s^{2}+(c_{011}\alpha_{1}-c_{101}\alpha_{2})s
	+(c_{021}\alpha_{1}^{2}-c_{111}\alpha_{1}\alpha_{2}
	+c_{201}\alpha_{2}^{2})=0.
	$$
	Let $\omega\in\mathbf{C}$ be a primitive $3$-rd root of unity.
	By the transformation
	$$
	\begin{pmatrix}
	 x_{1}\\
	 x_{2}
	\end{pmatrix}
	\longmapsto
	\begin{pmatrix}
	 \frac{\alpha_{1}\beta_{2}\gamma_{1}
	 +\alpha_{1}\beta_{1}\gamma_{2}\omega
	 +\alpha_{2}\beta_{1}\gamma_{1}\omega^{2}}
	 {\sqrt[3]{d}}
	 &\frac{-\alpha_{2}\beta_{1}\gamma_{2}
	 -\alpha_{2}\beta_{2}\gamma_{1}\omega
	 -\alpha_{1}\beta_{2}\gamma_{2}\omega^{2}}
	 {\sqrt[3]{d}}\\
	 \frac{-\alpha_{1}\beta_{2}\gamma_{1}
	 -\alpha_{2}\beta_{1}\gamma_{1}\omega
	 -\alpha_{1}\beta_{1}\gamma_{2}\omega^{2}}
	 {\sqrt[3]{d}}
	 &\frac{\alpha_{2}\beta_{1}\gamma_{2}
	 +\alpha_{1}\beta_{2}\gamma_{2}\omega
	 +\alpha_{2}\beta_{2}\gamma_{1}\omega^{2}}
	 {\sqrt[3]{d}}
	\end{pmatrix}
	\begin{pmatrix}
	 x_{1}\\
	 x_{2}
	\end{pmatrix},
	$$
	where
	\begin{align*}
	 d&=
	 \det{
	 \begin{pmatrix}
	  \alpha_{1}\beta_{2}\gamma_{1}
	  +\alpha_{1}\beta_{1}\gamma_{2}\omega
	  +\alpha_{2}\beta_{1}\gamma_{1}\omega^{2}
	  &-\alpha_{2}\beta_{1}\gamma_{2}
	  -\alpha_{2}\beta_{2}\gamma_{1}\omega
	  -\alpha_{1}\beta_{2}\gamma_{2}\omega^{2}\\
	  -\alpha_{1}\beta_{2}\gamma_{1}
	  -\alpha_{2}\beta_{1}\gamma_{1}\omega
	  -\alpha_{1}\beta_{1}\gamma_{2}\omega^{2}
	  &\alpha_{2}\beta_{1}\gamma_{2}
	  +\alpha_{1}\beta_{2}\gamma_{2}\omega
	  +\alpha_{2}\beta_{2}\gamma_{1}\omega^{2}
	 \end{pmatrix}}\\
	 &=(\omega-\omega^{2})(\alpha_{1}\beta_{2}-\alpha_{2}\beta_{1})
	 (\beta_{1}\gamma_{2}-\beta_{2}\gamma_{1})
	 (\gamma_{1}\alpha_{2}-\gamma_{2}\alpha_{1})
	 \neq0,
	\end{align*}
	$F(x)$ is normalized to satisfy
	$(c_{300},c_{210},c_{120},c_{020})
	=(\frac{1}{3},0,0,\frac{1}{3})$.
 \end{enumerate}
\end{proof}
\begin{proof}[Proof of Theorem~\ref{div}]
 For $p\in{X}$, by Lemma~\ref{nor},
 we may assume that $p=[1:0:0:0]$,
 $c_{000}=c_{100}=c_{010}=0$ and $c_{001}=1$.
 Then
 $$
 X\setminus\{x_{0}\neq0\}\simeq
 \{(\xi_{1},\xi_{2},\xi_{3})\in\mathbf{C}^{3}\mid
 F(1,\xi_{1},\xi_{2},\xi_{3})=0\},
 $$
 and $(\xi_{1},\xi_{2})$ gives a local coordinate of $X$ at $p$ because
 $\frac{\partial{F}}{\partial{x_{3}}}(p)=c_{001}\neq0$.
 For $[s_{1}:s_{2}]\in\mathbf{P}^{1}$, we set a line on $\mathbf{P}^{3}$
 by
 $$
 L_{[s_{1}:s_{2}]}=\{[x_{0}:\dots:x_{3}]\in\mathbf{P}^{3}\mid
 s_{1}x_{2}=s_{2}{x_{1}},\ x_{3}=0\},
 $$
 which intersects $X$ at $p$ with multiplicity $\geq2$.
 For $0\leq{i}\leq{3}$, we set a polynomial by
 $$
 f_{i}(\xi_{1},\xi_{2},\xi_{3},\zeta_{2},\zeta_{3})
 =F_{i}(1,\xi_{1},\xi_{2},\xi_{3},1,\zeta_{2},\zeta_{3}),
 $$
 where $F_{i}(x,z)$ is the polynomial defined in $(\ref{defeq})$.
 Then $Y$ is locally defined by these polynomials on a neighborhood of
 $(p,L_{[1:0]})\in{Y}$;
 $$
 Y\simeq\{(\xi_{1},\xi_{2},\xi_{3},\zeta_{2},\zeta_{3})\in\mathbf{C}^{5}
 \mid
 f_{0}(\xi,\zeta)=f_{1}(\xi,\zeta)=f_{2}(\xi,\zeta)=0\}.
 $$
 In order to give a local coordinate of $Y$, we divide the case into
 three types.
 \begin{enumerate}
  \item \label{c1}
	The case where $\phi^{-1}(p)$ is a set of distinct two
	points.
	By Lemma~\ref{nor}, we may assume that
	$c_{000}=c_{100}=c_{010}=c_{200}=c_{020}=0$
	and $c_{001}=c_{110}=1$.
	Then we have
	$\phi^{-1}(p)=\{(p,L_{[1:0]}),(p,L_{[0:1]})\}$.
	Since
	$$
	\begin{vmatrix}
	 \frac{\partial{f_{0}}}{\partial{\xi_{3}}}(0,0,0,0,0)&
	 \frac{\partial{f_{0}}}{\partial{\zeta_{2}}}(0,0,0,0,0)&
	 \frac{\partial{f_{0}}}{\partial{\zeta_{3}}}(0,0,0,0,0)\\
	 \frac{\partial{f_{1}}}{\partial{\xi_{3}}}(0,0,0,0,0)&
	 \frac{\partial{f_{1}}}{\partial{\zeta_{2}}}(0,0,0,0,0)&
	 \frac{\partial{f_{1}}}{\partial{\zeta_{3}}}(0,0,0,0,0)\\
	 \frac{\partial{f_{2}}}{\partial{\xi_{3}}}(0,0,0,0,0)&
	 \frac{\partial{f_{2}}}{\partial{\zeta_{2}}}(0,0,0,0,0)&
	 \frac{\partial{f_{2}}}{\partial{\zeta_{3}}}(0,0,0,0,0)
	\end{vmatrix}
	=
	\begin{vmatrix}
	 1&0&0\\
	 c_{101}&0&1\\
	 c_{201}&1&c_{101}
	\end{vmatrix}
	=-1\neq0,
	$$
	$(\xi_{1},\xi_{2})$ gives a local coordinate of $Y$ at
	$(p,L_{[1:0]})$ and $\phi$ is a local isomorphism in a
	neighborhood of $(p,L_{[1:0]})$.
	When $L_{[1:0]}$ is contained in $X$,
	$L_{[1:0]}^{+}\subset{Y}$ is locally isomorphic to
	$\{(\xi_{1},\xi_{2})\mid\xi_{2}=0\}$, and
	when $L_{[0:1]}$ is contained in $X$,
	$L_{[0:1]}^{-}\subset{Y}$ is locally isomorphic to
	$\{(\xi_{1},\xi_{2})\mid\xi_{1}=0\}$.
	Hence, if $(p,L_{[1:0]})\in{L_{[1:0]}^{+}\cap{L_{[0:1]}^{-}}}$,
	then $L_{[1:0]}^{+}$ intersects $L_{[0:1]}^{-}$ transversally at
	$(p,L_{[1:0]})\in{Y}$.
	In the same way, we can see the picture of a neighborhood
	of $(p,L_{[0:1]})$.
  \item \label{c2}
	The case where $\phi^{-1}(p)$ is a point.
	By Lemma~\ref{nor}, we may assume that
	$c_{000}=c_{100}=c_{010}=c_{200}=c_{110}=0$
	and $c_{001}=c_{020}=c_{210}=1$.
	Then
	$\phi^{-1}(p)=\{(p,L_{[1:0]})\}$.
	Since
	$$
	\begin{vmatrix}
	 \frac{\partial{f_{0}}}{\partial{\xi_{2}}}(0,0,0,0,0)&
	 \frac{\partial{f_{0}}}{\partial{\xi_{3}}}(0,0,0,0,0)&
	 \frac{\partial{f_{0}}}{\partial{\zeta_{3}}}(0,0,0,0,0)\\
	 \frac{\partial{f_{1}}}{\partial{\xi_{2}}}(0,0,0,0,0)&
	 \frac{\partial{f_{1}}}{\partial{\xi_{3}}}(0,0,0,0,0)&
	 \frac{\partial{f_{1}}}{\partial{\zeta_{3}}}(0,0,0,0,0)\\
	 \frac{\partial{f_{2}}}{\partial{\xi_{2}}}(0,0,0,0,0)&
	 \frac{\partial{f_{2}}}{\partial{\xi_{3}}}(0,0,0,0,0)&
	 \frac{\partial{f_{2}}}{\partial{\zeta_{3}}}(0,0,0,0,0)
	\end{vmatrix}
	=
	\begin{vmatrix}
	 0&1&0\\
	 0&c_{101}&1\\
	 1&c_{201}&c_{101}
	\end{vmatrix}
	=1\neq0,
	$$
	there are holomorphic functions $\varphi_{2}(\xi_{1},\zeta_{2})$,
	$\varphi_{3}(\xi_{1},\zeta_{2})$
	and $\mu_{3}(\xi_{1},\zeta_{2})$
	on a neighborhood of $(\xi_{1},\zeta_{2})=(0,0)$
	such that
	$$
	\varphi_{2}(0,0)=0,\
	\varphi_{3}(0,0)=0,\
	\mu_{3}(0,0)=0
	$$
	and
	$$
	f_{i}(\xi_{1},\varphi_{2}(\xi_{1},\zeta_{2}),
	\varphi_{3}(\xi_{1},\zeta_{2}),
	\zeta_{2},\mu_{3}(\xi_{1},\zeta_{2}))=0
	$$
	for $0\leq{i}\leq{2}$.
	We remark that
	\begin{align*}
	 &\varphi_{2}(\xi_{1},\zeta_{2})
	 \equiv-3c_{300}\xi_{1}
	 +(-9c_{300}^{2}c_{101}^{2}+9c_{300}^{2}c_{120}c_{101}
	 +9c_{300}^{2}c_{201}-3c_{300}c_{101})\xi_{1}^{2}\\
	 &\hspace*{100pt}
	 +(-6c_{300}c_{101}+6c_{300}c_{120}-2)\xi_{1}\zeta_{2}
	 -\zeta_{2}^{2}\mod
	 (\xi_{1}^{3},\xi_{1}^{2}\zeta_{2},
	 \xi_{1}\zeta_{2}^{2},\zeta_{2}^{3}),\\
	 &\varphi_{3}(\xi_{1},\zeta_{2})
	 \equiv-9c_{300}^{2}\xi_{1}^{2}\mod
	 (\xi_{1}^{3},\xi_{1}^{2}\zeta_{2},
	 \xi_{1}\zeta_{2}^{2},\zeta_{2}^{3}),\\
	 &\mu_{3}(\xi_{1},\zeta_{2})
	 \equiv(9c_{300}^{2}c_{101}-9c_{300}^{2}c_{120}+3c_{300})\xi_{1}^{2}
	 +6c_{300}\xi_{1}\zeta_{2}\mod
	 (\xi_{1}^{3},\xi_{1}^{2}\zeta_{2},
	 \xi_{1}\zeta_{2}^{2},\zeta_{2}^{3}).
	\end{align*}
	Then $(\xi_{1},\zeta_{2})$ is a local coordinate of $Y$ at
	$(p,L_{[1:0]})$, and $R=D$ is locally isomorphic to
	$\{(\xi_{1},\zeta_{2})\mid
	\frac{\partial\varphi_{2}}{\partial{\zeta_{2}}}
	(\xi_{1},\zeta_{2})=0\}$,
	and it is nonsingular at $(p,L_{[1:0]})$ because
	$\frac{\partial^{2}\varphi_{2}}{\partial{\zeta_{2}}^{2}}
	(0,0)=-2\neq0$.
	There is a holomorphic function $\sigma(\xi_{1})$ on a
	neighborhood of $\xi_{1}=0$ such that $\sigma(0)=0$ and
	$
	\frac{\partial\varphi_{2}}{\partial{\zeta_{2}}}
	(\xi_{1},\sigma(\xi_{1}))=0.
	$
	Then $B\subset{X}$ is locally isomorphic to
	$\{(\xi_{1},\xi_{2})\mid
	\xi_{2}=\varphi_{2}(\xi_{1},\sigma(\xi_{1}))\}$,
	and it is nonsingular at $p$.
	When $L_{[1:0]}$ is contained in $X$, we have
	$c_{300}=0$ and there is a holomorphic function
	$\eta_{2}(\xi_{1},\zeta_{2})$ such that
	$\varphi_{2}(\xi_{1},\zeta_{2})=\zeta_{2}
	\eta_{2}(\xi_{1},\zeta_{2})$.
	Then $L_{[1:0]}^{+}\subset{Y}$ is locally isomorphic to
	$\{(\xi_{1},\zeta_{2})\mid\zeta_{2}=0\}$, and
	$L_{[1:0]}^{-}\subset{Y}$ is locally isomorphic to
	$\{(\xi_{1},\zeta_{2})\mid\eta_{2}(\xi_{1},\zeta_{2})=0\}$.
	Since
	$$
	\begin{pmatrix}
	 \frac{\partial^{2}\varphi_{2}}
	 {\partial{\xi_{1}}\partial{\zeta_{2}}}(0,0)
	 &\frac{\partial\zeta_{2}}{\partial\xi_{1}}(0,0)
	 &\frac{\partial\eta_{2}}{\partial{\xi_{1}}}(0,0)\\
	 \frac{\partial^{2}\varphi_{2}}{\partial{\zeta_{2}}^{2}}(0,0)
	 &\frac{\partial\zeta_{2}}{\partial{\zeta_{2}}}(0,0)
	 &\frac{\partial\eta_{2}}{\partial\zeta_{2}}(0,0)
	\end{pmatrix}
	=
	\begin{pmatrix}
	 -2&0&-2\\
	 -2&1&-1
	\end{pmatrix},
	$$
	$D$ intersects $L_{[1:0]}^{+}$ and $L_{[1:0]}^{-}$ transversally,
	and $L_{[1:0]}^{+}$ intersects $L_{[1:0]}^{-}$ transversally at
	$(p,L_{[1:0]})\in{Y}$.
	Since $L$ is locally isomorphic to
	$\{(\xi_{1},\xi_{2})\mid\xi_{2}=0\}$ and
	$$
	\begin{cases}
	 \varphi_{2}(\xi_{1},\sigma(\xi_{1}))
	 \vert_{\xi_{1}=0}
	 =0,\\
	 \frac{d}{d{\xi_{1}}}(\varphi_{2}(\xi_{1},\sigma(\xi_{1})))
	 \vert_{\xi_{1}=0}
	 =0,\\
	 \frac{d^{2}}{d{\xi_{1}}^{2}}(\varphi_{2}(\xi_{1},\sigma(\xi_{1})))
	 \vert_{\xi_{1}=0}
	 =2\neq0,
	\end{cases}
	$$
	$L$ intersects $B$ at $p$ with multiplicity $2$.
  \item \label{c3}
	The case where $\phi^{-1}(p)\simeq{\mathbf{P}^{1}}$.
	By Lemma~\ref{nor}, we may assume that
	$c_{000}=c_{100}=c_{010}=c_{200}=c_{110}=c_{020}=c_{210}=c_{120}=0$
	and $c_{001}=3c_{300}=3c_{030}=1$.
	Since
	\begin{multline*}
	\begin{vmatrix}
	 \frac{\partial{f_{0}}}{\partial{\xi_{1}}}(0,0,0,\zeta_{2},0)&
	 \frac{\partial{f_{0}}}{\partial{\xi_{3}}}(0,0,0,\zeta_{2},0)&
	 \frac{\partial{f_{0}}}{\partial{\zeta_{3}}}(0,0,0,\zeta_{2},0)\\
	 \frac{\partial{f_{1}}}{\partial{\xi_{1}}}(0,0,0,\zeta_{2},0)&
	 \frac{\partial{f_{1}}}{\partial{\xi_{3}}}(0,0,0,\zeta_{2},0)&
	 \frac{\partial{f_{1}}}{\partial{\zeta_{3}}}(0,0,0,\zeta_{2},0)\\
	 \frac{\partial{f_{2}}}{\partial{\xi_{1}}}(0,0,0,\zeta_{2},0)&
	 \frac{\partial{f_{2}}}{\partial{\xi_{3}}}(0,0,0,\zeta_{2},0)&
	 \frac{\partial{f_{2}}}{\partial{\zeta_{3}}}(0,0,0,\zeta_{2},0)
	\end{vmatrix}\\
	 =
	\begin{vmatrix}
	 0&1&0\\
	 0&c_{101}+c_{011}\zeta_{2}&1\\
	 1&c_{201}+c_{111}\zeta_{2}+c_{021}\zeta_{2}^{2}&
	 c_{101}+c_{011}\zeta_{2}
	\end{vmatrix}
	=1\neq0,
	\end{multline*}
	there are holomorphic functions $\varphi_{1}(\xi_{2},\zeta_{2})$,
	$\varphi_{3}(\xi_{2},\zeta_{2})$
	and $\mu_{3}(\xi_{2},\zeta_{2})$
	on a neighborhood of $\{(\xi_{2},\zeta_{2})\mid\xi_{2}=0\}$
	such that
	$$
	\varphi_{1}(0,\zeta_{2})=0,\
	\varphi_{3}(0,\zeta_{2})=0,\
	\mu_{3}(0,\zeta_{2})=0
	$$
	and
	$$
	f_{i}(\varphi_{1}(\xi_{2},\zeta_{2}),\xi_{2},
	\varphi_{3}(\xi_{2},\zeta_{2}),
	\zeta_{2},\mu_{3}(\xi_{2},\zeta_{2}))=0
	$$
	for $0\leq{i}\leq2$.
	We remark that
	\begin{align*}
	 &\varphi_{1}(\xi_{2},\zeta_{2})
	 \equiv
	 -\zeta_{2}^{2}\xi_{2}+(c_{101}\zeta_{2}+c_{011}\zeta_{2}^{2}
	 +c_{101}\zeta_{2}^{4}+c_{011}\zeta_{2}^{5})\xi_{2}^{2}
	 \mod(\xi_{2}^{3}),\\
	 &\varphi_{3}(\xi_{2},\zeta_{2})\equiv0\mod(\xi_{2}^{3}),\\
	 &\mu_{3}(\xi_{2},\zeta_{2})\equiv
	 (-\zeta_{2}-\zeta_{2}^{4})\xi_{2}^{2}\mod(\xi_{2}^{3}).
	\end{align*}
	There is a holomorphic function $\eta_{1}(\xi_{2},\zeta_{2})$
	such that
	$$
	\varphi_{1}(\xi_{2},\zeta_{2})
	=\xi_{2}\eta_{1}(\xi_{2},\zeta_{2}).
	$$
	Since $R$ is locally isomorphic to
	$\{(\xi_{2},\zeta_{2})\mid
	\frac{\partial\varphi_{1}}{\partial{\zeta_{2}}}
	(\xi_{2},\zeta_{2})=0\}$,
	$E$ is locally isomorphic to
	$\{(\xi_{2},\zeta_{2})\mid\xi_{2}=0\}$
	and $D$ is locally isomorphic to
	$\{(\xi_{2},\zeta_{2})\mid
	\frac{\partial\eta_{1}}{\partial{\zeta_{2}}}
	(\xi_{2},\zeta_{2})=0\}$.
	We remark that
	$L_{[1:\lambda]}\subset{X}$ if and only if  $\lambda^{3}+1=0$.
	Hence $p$ is an Eckardt point on $X$.
	We assume that $\lambda^{3}+1=0$.
	Then $L_{[1:\lambda]}$ is locally isomorphic to
	$\{(\xi_{1},\xi_{2})\mid\xi_{2}=\lambda\xi_{1}\}$
	and
	$\phi^{*}L_{[1:\lambda]}$
	is locally isomorphic to
	$\{(\xi_{2},\zeta_{2})\mid
	\xi_{2}=\lambda\varphi_{1}(\xi_{2},\zeta_{2})
	\}$,
	hence
	$L_{[1:\lambda]}^{+}+L_{[1:\lambda]}^{-}$
	is locally isomorphic to
	$\{(\xi_{2},\zeta_{2})\mid
	1=\lambda\eta_{1}(\xi_{2},\zeta_{2})\}$.
	Since $\eta_{1}(0,\zeta_{2})=-\zeta_{2}^{2}$,
	$$
	(0,\zeta_{2})\in{L_{[1:\lambda]}^{+}+L_{[1:\lambda]}^{-}}
	\Longleftrightarrow
	1=-\lambda\zeta_{2}^{2}
	\Longleftrightarrow
	\zeta_{2}^{2}=\lambda^{2}.
	$$
	Then $L_{[1:\lambda]}^{+}$ intersects $E$ transversally at
	$(p,L_{[1:\lambda]})$ by
	$$
	\frac{\partial}{\partial{\zeta_{2}}}(1-\lambda\eta_{1})
	\Big\vert_{(\xi_{2},\zeta_{2})=(0,\lambda)}=2\lambda^{2}\neq0,
	$$
	and $L_{[1:\lambda]}^{-}$ intersects $E$ transversally at
	$(p,L_{[1:-\lambda]})$ by
	$$
	\frac{\partial}{\partial{\zeta_{2}}}(1-\lambda\eta_{1})
	\Big\vert_{(\xi_{2},\zeta_{2})=(0,-\lambda)}=-2\lambda^{2}\neq0.
	$$
	Since
	$\frac{\partial\eta_{1}}{\partial{\zeta_{2}}}(0,\zeta_{2})
	=-2\zeta_{2}$,
	$$
	(0,\zeta_{2})\in{D}
	\Longleftrightarrow
	\zeta_{2}=0.
	$$
	Then $D$ intersects $E$ transversally at $(p,L_{[1:0]})$ by
	$$
	\frac{\partial^{2}\eta_{1}}{\partial{\zeta_{2}}^{2}}(0,0)=-2\neq0.
	$$
	There is a holomorphic function $\sigma(\xi_{2})$ on a
	neighborhood of $\xi_{2}=0$ such that $\sigma(0)=0$ and
	$\frac{\partial\eta_{1}}{\partial{\zeta_{2}}}
	(\xi_{2},\sigma(\xi_{2}))=0$.
	Then the image $B_{1}$ of the local component of $D$ at
	$(p,L_{[1:0]})$ by $\phi$ is locally isomorphic to
	$\{(\xi_{1},\xi_{2})\mid\xi_{1}
	=\varphi_{1}(\xi_{2},\sigma(\xi_{2}))\}$.
	Since
	$\frac{\partial}{\partial\xi_{2}}
	(\varphi_{1}(\xi_{2},\sigma(\xi_{2})))\vert_{\xi_{2}=0}=0$,
	$B_{1}$ intersects $L_{[1:\lambda]}$ transversally at $p$.
	In the same way, we can show that $D$  intersects $E$
	transversally at $(p,L_{[0:1]})$, and there is a holomorphic
	function $\tau(\xi_{1})$ on a neighborhood of $\xi_{1}=0$ such
	that $\frac{d{\tau}}{d\xi_{1}}(0)=0$ and the image $B_{2}$ of
	the local component of $D$ at $(p,L_{[0:1]})$ by $\phi$
	is locally isomorphic to
	$\{(\xi_{1},\xi_{2})\mid\xi_{2}=\tau(\xi_{1})\}$.
	Then $B_{2}$ intersects $L_{[1:\lambda]}$ and $B_{1}$
	transversally at $p$.
	This implies that $B$ has a node at $p$, and $L_{[1:\lambda]}$
	intersects $B$ at $p$ with multiplicity $2$.
 \end{enumerate}
 By the above observation, we have
 $\phi^{*}L=L^{+}+L^{-}
 +\sum_{e\in{L\cap\Sing{(B)}}}\phi^{-1}(e)$
 for a line $L$ on $X$, and $B\cap{L}$ is a set of distinct two point
 because $(B.L)=4$.
 Hence we have
 \begin{align*}
  (L^{-}.\ L^{-})&=(L^{-}.\ \phi^{*}L-L^{+}
  -\sum_{e\in{L\cap\Sing{(B)}}}\phi^{-1}(e))\\
  &=(L.\ L)-(L^{-}.\ L^{+}+\sum_{e\in{L\cap\Sing{(B)}}}\phi^{-1}(e))
  =-1-2=-3.
 \end{align*}
 Each component of $E$ corresponds to an Eckardt point on $X$, and it is
 a $(-2)$-curve on $Y$, because $\phi$ is the canonical map of $Y$ by
 Proposition~\ref{ni}.
\end{proof}
\begin{remark}
 There are at most two Eckardt points on a line $L\subset{X}$, hence
 there are at most $18$ Eckardt points on $X$.
 If $X$ has $18$ Eckardt points, then $X$ is isomorphic to the Fermat
 cubic surface \cite{s}.
\end{remark}
\begin{remark}
 Let $\phi':Y'\rightarrow{X}$ be the finite double cover of $X$ branched
 along $B$.
 Then $Y'$ may have ordinary double points, and $Y$ is the minimal
 resolution of $Y'$,
\end{remark}
\begin{remark}\label{int}
 By Theorem~$\ref{div}$, for lines $L_{1},L_{2},L$ on $X$ and Eckardt
 points $e_{1},e_{2},e$ on $X$, the intersection numbers on $Y$ are
 computed by
 \begin{align*}
  &(L_{1}^{+}.L_{2}^{+})=(L_{1}^{-}.L_{2}^{-})=
  \begin{cases}
   0&\text{if $L_{1}\neq{L_{2}}$},\\
   -3&\text{if $L_{1}={L_{2}}$},
  \end{cases}\\
  &(L_{1}^{+}.L_{2}^{-})=
  \begin{cases}
   0&\text{if $L_{1}\cap{L_{2}}=\emptyset$},\\
   1&\text{if $L_{1}\cap{L_{2}}$ is a point which is not an Eckardt
   point},\\
   0&\text{if $L_{1}\cap{L_{2}}$ is a point which is an Eckardt
   point},\\
   0&\text{if $L_{1}=L_{2}$ and there are two Eckardt points on
   $L_{1}=L_{2}$},\\
   1&\text{if $L_{1}=L_{2}$ and there is only one Eckardt point on
   $L_{1}=L_{2}$},\\
   2&\text{if $L_{1}=L_{2}$ and there are no Eckardt points on
   $L_{1}=L_{2}$},
  \end{cases}\\
  &(\phi^{-1}(e_{1}).\phi^{-1}(e_{2}))=
  \begin{cases}
   0&\text{if $e_{1}\neq{e_{2}}$}.\\
   -2&\text{if $e_{1}=e_{2}$},
  \end{cases}\\
  &(L^{+}.\phi^{-1}(e))=(L^{-}.\phi^{-1}(e))=
  \begin{cases}
   0&\text{if $e\notin{L}$},\\
   1&\text{if $e\in{L}$}.
  \end{cases}
 \end{align*}
\end{remark}
\begin{proposition}
 Any $(-2)$-curve on $Y$ is a component of $E$, and any $(-3)$-curve
 on $Y$ is a component of $Y_{\infty}+Y_{\infty}^{-}$.
\end{proposition}
\begin{proof}
 Let $C$ be a $(-2)$-curve on $Y$.
 Since
 $(\phi_{*}C.\ \mathcal{O}_{\mathbf{P}^{3}}(1)\vert_{X})
 =(C.\ K_{Y})=0$,
 the image of $C$ by the morphism $\phi$ is a point on $X$, hence
 $C$ is a component of $E$.
 Let $C$ be a $(-3)$-curve on $Y$.
 Since
 $(\phi_{*}C.\ \mathcal{O}_{\mathbf{P}^{3}}(1)\vert_{X})
 =(C.\ K_{Y})=1$,
 the image of $C$ by the morphism $\phi$ is a line on $X$, hence
 $C$ is is a component of $Y_{\infty}+Y_{\infty}^{-}$.
\end{proof}
\begin{remark}
 We can check that the divisor $Y_{\infty}+Y_{\infty}^{-}$ is connected.
 Hence, if a divisor $W$ on $Y$ is a disjoint union of irreducible
 components of $Y_{\infty}+Y_{\infty}^{-}$, and $W$ contains a component
 of $\phi^{*}L$ for any line $L$ on $X$, then $W=Y_{\infty}$ or
 $W=Y_{\infty}^{-}$.
\end{remark}
Let $\psi=\varPsi\vert_{Y}:
Y\rightarrow{Z}=Z_{3}\subset\lambda(\mathbf{P}^{3})$ be the
second projection in Remark~\ref{z3},
and let
$[\mathcal{O}_{Z}(1)]\in
H^{2}(Z,\mathbf{Z})$
be the class of a hyperplane section by the
Pl\"{u}cker embedding $\Lambda(\mathbf{P}^{3})\subset\mathbf{P}^{5}$.
Let $Z_{\infty}$ be the set of
all lines on the cubic surface $X$.
For a line $L_{0}\in{Z_{\infty}}$, we set
$Z_{\infty}(L_{0})=\{L\in{Z_{\infty}}\mid{L_{0}\neq{L}},\
L_{0}\cap{L}\neq\emptyset\}$,
which is a set of $10$ lines.
\begin{proposition}\label{rel}
 There are the following relations in the N\'{e}ron-Severi group
 $\NS{(Y)}$$:$
 \begin{equation}\label{rel1}
 \psi^{*}[\mathcal{O}_{Z}(3)]=\phi^{*}[\mathcal{O}_{X}(3)]
 +\sum_{L\in{Z_{\infty}}}L^{+}
 \end{equation}
 and
 \begin{equation}\label{rel2}
 \psi^{*}[\mathcal{O}_{Z}(1)]=
 3\phi^{*}L_{0}-L_{0}^{+}+\sum_{L\in{Z_{\infty}(L_{0})}}L^{+}
 \end{equation}
 for any line $L_{0}\in{Z_{\infty}}$.
\end{proposition}
\begin{proof}
 Since $Y_{\infty}=\coprod_{L\in{Z_{\infty}}}L^{+}$, the relation
 $(\ref{rel1})$ is given by
 $$
 \mathcal{O}_{Y}(Y_{\infty})\simeq
 \mathcal{S}^{\otimes{3}}\vert_{Y}
 \simeq\psi^{*}\mathcal{O}_{Z}(3)\otimes
 \phi^{*}\mathcal{O}_{X}(-3).
 $$
 For $L_{0}\in\Lambda(\mathbf{P}^{3})$,
 $$
 H_{L_{0}}=\{L\in\Lambda(\mathbf{P}^{3})\mid{L_{0}\cap{L}\neq\emptyset}\}
 $$
 is a hyperplane section by the Pl\"{u}cker embedding
 $\Lambda(\mathbf{P}^{3})\subset\mathbf{P}^{5}$.
 We prove that
 $$
 \psi^{*}H_{L_{0}}
 =2L_{0}^{+}+3L_{0}^{-}+3\sum_{e\in{L_{0}}\cap\Sing{(B)}}\phi^{-1}(e)
 +\sum_{L\in{Z_{\infty}(L_{0})}}L^{+}
 $$
 for $L_{0}\in{Z_{\infty}}$.
 It gives the relation~$(\ref{rel2})$ by the relation in
 Theorem~$\ref{div}$.
 For $(p,L)\in\psi^{-1}(H_{L_{0}})\subset{Y}$, if $p\in{L_{0}}$
 then
 $$
 (p,L)\in\phi^{-1}(L_{0})
 ={L_{0}^{+}\cup{L_{0}^{-}}\cup
 \bigcup_{e\in{L_{0}}\cap\Sing{(B)}}\phi^{-1}(e)},
 $$
 and if $p\notin{L_{0}}$ then $L\subset{X}$.
 Hence the support of $\psi^{*}H_{L_{0}}$ is
 $$
 \psi^{-1}(H_{L_{0}})=
 {L_{0}^{+}\cup{L_{0}^{-}}\cup
 \bigcup_{e\in{L_{0}}\cap\Sing{(B)}}\phi^{-1}(e)}
 \cup
 \bigcup_{L\in{Z_{\infty}(L_{0})}}L^{+}.
 $$
 We compute the multiplicity of each component.
 \begin{enumerate}
  \item The case where there are no Eckardt points on the line $L_{0}$.
	We set integers $a_{+}$, $a_{-}$ and $a_{L}$ by
	$$
	\psi^{*}[\mathcal{O}_{Z}(1)]=
	\psi^{*}H_{L_{0}}
	=a_{+}L_{0}^{+}+a_{-}L_{0}^{-}
	+\sum_{L\in{Z_{\infty}(L_{0})}}a_{L}L^{+}.
	$$
	Since $(\psi^{*}[\mathcal{O}_{Z}(1)].\ L^{+})=0$ for
	$L\in{Z_{\infty}}$,
	$$
	0=(\psi^{*}H_{L_{0}}.\ L^{+})
	=
	\begin{cases}
	 -3a_{+}+2a_{-}&\text{if $L=L_{0}$},\\
	 a_{-}-3a_{L}&\text{if $L\in{Z_{\infty}(L_{0})}$}.
	\end{cases}
	$$
	By the relation $(\ref{rel1})$,
	$$
	(\psi^{*}[\mathcal{O}_{Z}(3)].\ L_{0}^{-})
	=(\phi^{*}[\mathcal{O}_{X}(3)].\ L_{0}^{-})
	+(L_{0}^{+}.\ L_{0}^{-})
	+\sum_{L\in{Z_{\infty}(L_{0})}}(L^{+}.\ L_{0}^{-})
	=3+2+10,
	$$
	hence we have
	$$
	5=(\psi^{*}H_{L_{0}}.\ L_{0}^{-})
	=2a_{+}-3a_{-}+\sum_{L\in{Z_{\infty}(L_{0})}}a_{L}.
	$$
	These equations imply that $a_{+}=2$, $a_{-}=3$ and $a_{L}=1$
	for $L\in{Z_{\infty}(L_{0})}$.
  \item The case where there is only one Eckardt point $e$ on the line
	$L_{0}$.
	We denote by $Z_{\infty}(e,L_{0})\subset{Z_{\infty}}(L_{0})$ the
	set of two lines through the point $e$.
	We set integers $a_{+}$, $a_{-}$, $b$ and $a_{L}$ by
	$$
	\psi^{*}[\mathcal{O}_{Z}(1)]=
	\psi^{*}H_{L_{0}}
	=a_{+}L_{0}^{+}+a_{-}L_{0}^{-}+b\phi^{-1}(e)
	+\sum_{L\in{Z_{\infty}(L_{0})}}a_{L}L^{+}.
	$$
	Since $(\psi^{*}[\mathcal{O}_{Z}(1)].\ L^{+})=0$ for
	$L\in{Z_{\infty}}$,
	$$
	0=(\psi^{*}H_{L_{0}}.\ L^{+})
	=
	\begin{cases}
	 -3a_{+}+a_{-}+b&\text{if $L=L_{0}$},\\
	 a_{-}-3a_{L}&
	 \text{if ${L}\in{Z_{\infty}(L_{0})}\setminus
	 Z_{\infty}(e,L_{0})$},\\
	 b-3a_{L}&
	 \text{if ${L}\in{Z_{\infty}(e,L_{0})}$}.
	\end{cases}
	$$
	By the relation $(\ref{rel1})$,
	$$
	(\psi^{*}[\mathcal{O}_{Z}(3)].\ L_{0}^{-})
	=(\phi^{*}[\mathcal{O}_{X}(3)].\ L_{0}^{-})
	+(L_{0}^{+}.\ L_{0}^{-})
	+\sum_{L\in{Z_{\infty}(L_{0})}}(L^{+}.\ L_{0}^{-})
	=3+1+8
	$$
	and
	\begin{align*}
	(\psi^{*}[\mathcal{O}_{Z}(3)].\ \phi^{-1}(e))
	&=(\phi^{*}[\mathcal{O}_{X}(3)].\ \phi^{-1}(e))
	+(L_{0}^{+}.\ \phi^{-1}(e))
	+\sum_{L\in{Z_{\infty}(L_{0})}}(L^{+}.\ \phi^{-1}(e))\\
	&=0+1+2,
	\end{align*}
	hence we have
	$$
	4=(\psi^{*}H_{L_{0}}.\ L_{0}^{-})
	=a_{+}-3a_{-}+b
	+\sum_{{L}\in{Z_{\infty}(L_{0})}\setminus
	Z_{\infty}(e,L_{0})}a_{L}
	$$
	and
	$$
	1=(\psi^{*}H_{L_{0}}.\ \phi^{-1}(e))
	=a_{+}+a_{-}-2b+\sum_{{L}\in{Z_{\infty}(e,L_{0})}}a_{L}.
	$$
	These equations imply that $a_{+}=2$, $a_{-}=3$, $b=3$ and
	$a_{L}=1$ for $L\in{Z_{\infty}(L_{0})}$.
  \item The case where there are two Eckardt points $e_{1}$, $e_{2}$ on
	the line $L_{0}$.
	We set integers $a_{+}$, $a_{-}$, $b_{1}$, $b_{2}$ and $a_{L}$
	by
	$$
	\psi^{*}[\mathcal{O}_{Z}(1)]=
	\psi^{*}H_{L_{0}}
	=a_{+}L_{0}^{+}+a_{-}L_{0}^{-}
	+b_{1}\phi^{-1}(e_{1})+b_{2}\phi^{-1}(e_{2})
	+\sum_{L\in{Z_{\infty}(L_{0})}}a_{L}L^{+}.
	$$
	Since $(\psi^{*}[\mathcal{O}_{Z}(1)].\ L^{+})=0$ for
	$L\in{Z_{\infty}}$,
	\begin{align*}
	0&=(\psi^{*}H_{L_{0}}.\ L^{+})\\
	&=
	\begin{cases}
	 -3a_{+}+b_{1}+b_{2}&\text{if $L=L_{0}$},\\
	 a_{-}-3a_{L}&
	 \text{if ${L}\in{Z_{\infty}(L_{0})}\setminus
	 \bigl(Z_{\infty}(e_{1},L_{0}){\cup}Z_{\infty}(e_{2},L_{0})
	 \bigr)$},\\
	 b_{i}-3a_{L}&
	 \text{if ${L}\in{Z_{\infty}(e_{i},L_{0})}$}.
	\end{cases}
	\end{align*}
	By the relation $\ref{rel1}$,
	$$
	(\psi^{*}[\mathcal{O}_{Z}(3)].\ L_{0}^{-})
	=(\phi^{*}[\mathcal{O}_{X}(3)].\ L_{0}^{-})
	+(L_{0}^{+}.\ L_{0}^{-})
	+\sum_{L\in{Z_{\infty}(L_{0})}}(L^{+}.\ L_{0}^{-})
	=3+0+6
	$$
	and
	\begin{align*}
	&(\psi^{*}[\mathcal{O}_{Z}(3)].\ \phi^{-1}(e_{i}))\\
	=&(\phi^{*}[\mathcal{O}_{X}(3)].\ \phi^{-1}(e_{i}))
	+(L_{0}^{+}.\ \phi^{-1}(e_{i}))
	+\sum_{L\in{Z_{\infty}(L_{0})}}(L^{+}.\ \phi^{-1}(e_{i}))\\
	=&0+1+2,
	\end{align*}
	hence we have
	$$
	3=(\psi^{*}H_{L_{0}}.\ L_{0}^{-})
	=-3a_{-}+b_{1}+b_{2}
	+\sum_{{L}\in{Z_{\infty}(L_{0})}
	\setminus(Z_{\infty}(e_{1},L_{0}){\cup}Z_{\infty}(e_{2},L_{0})
	)}a_{L}
	$$
	and
	$$
	1=(\psi^{*}H_{L_{0}}.\ \phi^{-1}(e_{i}))
	=a_{+}+a_{-}-2b_{i}+\sum_{{L}\in{Z_{\infty}(e_{i},L_{0})}}a_{L}.
	$$
	These equations imply that $a_{+}=2$, $a_{-}=3$, $b_{1}=3$,
	$b_{2}=3$ and $a_{L}=1$ for $L\in{Z_{\infty}(L_{0})}$.
 \end{enumerate}
\end{proof}
\section{Periods of cubic $3$-folds}\label{ct}
We review some works on cubic $3$-folds by Clemens-Griffiths \cite{cg}
and Tjurin \cite{t}.
Let $V\subset{\mathbf{P}^{4}}$ be a nonsingular cubic $3$-folds.
We define a subvariety $W$ of
$\mathbf{P}^{4}\times\Lambda(\mathbf{P}^{4})$
by
$$
W=\{(p,L)\in\mathbf{P}^{4}\times\Lambda(\mathbf{P}^{4})\mid
p\in{L}\subset{V}\},
$$
and we define a subvariety $S$ of $\Lambda(\mathbf{P}^{4})$ by
$$
S=\{L\in\Lambda(\mathbf{P}^{4})\mid
L\subset{V}\},
$$
which is a nonsingular surface and called the Fano surface of lines on
$V$.
The first projection $\phi:W\rightarrow{V}$ is a generically finite
morphism of degree $6$, and the second projection $\psi:W\rightarrow{S}$
is a $\mathbf{P}^{1}$-bundle.
\begin{theorem}[Clemens-Griffiths \cite{cg}, Theorem~11.19]\label{aj}
 The homomorphism
 $$
 \phi_{*}\circ\psi^{*}:
 H^{3}(S,\mathbf{Z})\longrightarrow
 H^{3}(V,\mathbf{Z})
 $$
 is an isomorphism of Hodge structures.
\end{theorem}
Let $J$ be the intermediate Jacobian of the Hodge structure
$H^{3}(V,\mathbf{Z})$.
Then the complex torus $J$ is a principally polarized abelian variety of
dimension $5$.
We denote by $\theta\in{H^{2}(J,\mathbf{Z})}$ the class of the
polarization.
Let $A$ be the Albanese variety of $S$, and $\iota:S\rightarrow{A}$ the
Albanese morphism.
By Theorem~$\ref{aj}$, there is a natural isomorphism $A\simeq{J}$ of
abelian varieties.
Let us denote by $\theta\in{H^{2}(A,\mathbf{Z})}$ the corresponding
principal polarization on $A$.
The primitive part of
$H^{2}(A,\mathbf{Z})$ is defined as the space
$$
H^{2}_{\mathrm{prim}}(A,\mathbf{Z})
=\Ker{\bigl(\theta^{\cup4}:H^{2}(A,\mathbf{Z})\longrightarrow
H^{10}(A,\mathbf{Z});\
\alpha\longmapsto\theta^{\cup4}\cup\alpha\bigr)},
$$
and the primitive part of
$H^{2}(S,\mathbf{Z})$ is defined as the space
$$
H^{2}_{\mathrm{prim}}(S,\mathbf{Z})
=\Ker{\bigl([\mathcal{O}_{S}(1)]:H^{2}(S,\mathbf{Z})\longrightarrow
H^{4}(S,\mathbf{Z});\
\beta\longmapsto[\mathcal{O}_{S}(1)]\cup\beta\bigr)},
$$
where
$[\mathcal{O}_{S}(1)]\in
H^{2}(S,\mathbf{Z})$
is the class of a hyperplane section by the
Pl\"{u}cker embedding $\Lambda(\mathbf{P}^{4})\subset\mathbf{P}^{9}$.
We define a symmetric form on $H^{2}(A,\mathbf{Z})$ by
$$
\langle\ ,\ \rangle_{A}:
H^{2}(A,\mathbf{Z})\times
H^{2}(A,\mathbf{Z})\longrightarrow
\mathbf{Z};\
(\alpha_{1},\alpha_{2})\longmapsto
\deg{\Bigl(\Bigl(\frac{\theta^{\cup3}}{3!}
\cup\alpha_{1}\cup\alpha_{2}\Bigr)\cap[A]\Bigr)},
$$
and a symmetric form on
$H^{2}(S,\mathbf{Z})$ by
$$
\langle\ ,\ \rangle_{S}:
H^{2}(S,\mathbf{Z})\times
H^{2}(S,\mathbf{Z})\longrightarrow
\mathbf{Z};\
(\beta_{1},\beta_{2})\longmapsto
\deg{((\beta_{1}\cup\beta_{2})\cap[S])}.
$$
We remark that these symmetric forms give polarizations of Hodge
structures on the primitive part $H^{2}_{\mathrm{prim}}(A,\mathbf{Z})$
and $H^{2}_{\mathrm{prim}}(S,\mathbf{Z})$.
\begin{proposition}\label{ht}
 The homomorphism
 $\iota^{*}:H^{2}(A,\mathbf{Z})\rightarrow
 H^{2}(S,\mathbf{Z})$
 induces the isomorphism
 $$
 \bigl(H^{2}_{\mathrm{prim}}(A,\mathbf{Z}),\
 \langle\ ,\ \rangle_{A}\bigr)\simeq
 \bigl(H^{2}_{\mathrm{prim}}(S,\mathbf{Z}),\
 \langle\ ,\ \rangle_{S}\bigr)
 $$
 of polarized Hodge structures.
\end{proposition}
\begin{proof}
 By \cite[Lemma~9.13 and (10.14)]{cg}, the homomorphism
 $
 \iota^{*}:
 H^{2}(A,\mathbf{Z})\rightarrow
 H^{2}(S,\mathbf{Z})
 $
 is injective with a finite cokernel.
 By \cite[(2.3.5)]{c}, the homology group
 $H_{1}(S,\mathbf{Z})$ has no torsion element, and
 the cokernel of
 $
 \iota_{*}:
 H_{2}(S,\mathbf{Z})\rightarrow
 H_{2}(A,\mathbf{Z})
 $
 is isomorphic to $\mathbf{Z}/2\mathbf{Z}$.
 Hence $H^{2}(S,\mathbf{Z})$ has no torsion element, and the cokernel of
 $\iota^{*}:H^{2}(A,\mathbf{Z})\rightarrow
 H^{2}(S,\mathbf{Z})$
 is isomorphic to $\mathbf{Z}/2\mathbf{Z}$.
 Since $[\iota(S)]=\frac{\theta^{\cup3}}{6}\in{H^{6}(A,\mathbf{Z})}$ by
 \cite[Proposition~13.1]{cg},
 we have
 $$
 \iota_{*}((\iota^{*}\alpha_{1}\cup\iota^{*}\alpha_{2})\cap[S])
 =(\alpha_{1}\cup\alpha_{2})\cap\iota_{*}[S]
 =(\alpha_{1}\cup\alpha_{2})\cap
 \Bigl(\frac{\theta^{\cup3}}{6}\cap[A]\Bigr)\\
 =\Bigl(\frac{\theta^{\cup{3}}}{6}\cup\alpha_{1}\cup\alpha_{2}\Bigr)\cap[A]
 $$
 for $\alpha_{1},\alpha_{2}\in{H^{2}(A,\mathbf{Z})}$, hence the
 homomorphism $\iota^{*}$ is compatible with the symmetric forms.
 Let $\tau\in{H^{2}(S,\mathbf{Z})}$ be the class of an incidence divisor
 \cite[\textsection{2}]{cg}.
 Since $3\tau=[\mathcal{O}_{S}(1)]$ by \cite[\textsection{10}]{cg}, the
 primitive part
 $H^{2}_{\mathrm{prim}}(S,\mathbf{Z})$ is equal to the space orthogonal
 to $\tau$.
 Since $2\tau=\iota^{*}\theta$ by \cite[Lemma~11.27]{cg}, we have
 $$
 \iota_{*}((2\tau\cup\iota^{*}\alpha)\cap[S])
 =\iota_{*}((\iota^{*}\theta\cup\iota^{*}\alpha)\cap[S])
 =\Bigl(\frac{\theta^{\cup{4}}}{6}\cup\alpha\Bigr)\cap[A]
 $$
 for any $\alpha\in{H^{2}(A,\mathbf{Z})}$.
 Hence we have a commutative diagram of exact sequences
 $$
 \begin{matrix}
  &&0&&0&&&&\\
  &&\downarrow&&\downarrow&&&&\\
  &&H^{2}_{\mathrm{prim}}(A,\mathbf{Z})&&
  H^{2}_{\mathrm{prim}}(S,\mathbf{Z})&&&&\\
  &&\downarrow&&\downarrow&&&&\\
  0&\longrightarrow&
  H^{2}(A,\mathbf{Z})&\overset{\iota^{*}}{\longrightarrow}
  &H^{2}(S,\mathbf{Z})&\longrightarrow
  &\mathbf{Z}/2\mathbf{Z}&\longrightarrow&0\\
  &&
  \frac{\theta^{\cup4}}{12}\downarrow\qquad&&\quad\downarrow\tau
  &&&&\\
  &&
  \mathbf{Z}&=&\mathbf{Z}.&
  &&&
 \end{matrix}
 $$
 Since $\theta$ is a principal polarization, the image of the
 homomorphism
 $$
 \frac{\theta^{\cup4}}{12}:H^{2}(A,\mathbf{Z})\rightarrow\mathbf{Z};\
 \alpha\longmapsto\deg{\Bigl(\Bigl(
 \frac{\theta^{\cup4}}{12}\cup\alpha\Bigr)\cap[A]\Bigr)}
 $$
 is $2\mathbf{Z}$.
 And the image of the homomorphism
 $$
 \tau:H^{2}(S,\mathbf{Z})\rightarrow\mathbf{Z};\
 \alpha\longmapsto\deg{((\tau\cup\alpha)\cap[S])}
 $$
 is not contained in $2\mathbf{Z}$,
 because $\deg{(\tau^{\cup2}\cap[S])}=5\notin2\mathbf{Z}$ by
 \cite[(10.8)]{cg}.
 Hence $\tau:H^{2}(S,\mathbf{Z})\rightarrow\mathbf{Z}$
 is surjective, and
 $\iota^{*}:H^{2}_{\mathrm{prim}}(A,\mathbf{Z})\rightarrow
 H^{2}_{\mathrm{prim}}(S,\mathbf{Z})$
 is an isomorphism.
\end{proof}
\section{Periods of cubic surfaces}\label{cs}
Let $X\subset{\mathbf{P}^{3}}$ be a nonsingular cubic surface defined by
$F(x_{0},\dots,x_{3})\in\mathbf{C}[x_{0},\dots,x_{3}]$.
Let $V\subset\mathbf{P}^{4}$ be the cubic $3$-fold defined by
$F(x_{0},\dots,x_{3})+x_{4}^{3}\in\mathbf{C}[x_{0},\dots,x_{4}]$.
Then the projection
$$
\rho:V\rightarrow\mathbf{P}^{3};\
[x_{0}:\dots:x_{3}:x_{4}]\mapsto[x_{0}:\dots:x_{3}]
$$
is the triple Galois cover branched along the cubic surface $X$.
Let $S$ be the Fano surface of lines on $V$.
Then the Galois group $\Gal{(\rho)}\simeq\mathbf{Z}/3\mathbf{Z}$ of the
cover $\rho$ acts on the surface $S$.
\begin{lemma}\label{tc}
 Let $L$ be a line in $\mathbf{P}^{4}$.
 If $L$ is contained in $V$, then its image
 $\rho(L)\subset\mathbf{P}^{3}$ by $\rho$ is a line in $\mathbf{P}^{3}$,
 and it is contained in $X$ or intersects $X$ at only one point with
 multiplicity $3$.
\end{lemma}
\begin{proof}
 Let $H_{4}\subset\mathbf{P}^{4}$ be the hyperplane defined by the
 equation $x_{4}=0$.
 If $L$ is contained in $H_{4}\cap{V}$, then it is clear that $\rho(L)$
 is a line contained in $X$.
 We assume that $L\cap{H_{4}}$ is a point
 $[a_{0}:\dots:a_{3}:0]\in\mathbf{P}^{4}$.
 By taking a point $[b_{0}:\dots:b_{3}:1]\in{L\setminus{H_{4}}}$,
 the line $L$ is written as
 $$
 L=\{[a_{0}t_{0}+b_{0}t_{1}:\dots:a_{3}t_{0}+b_{3}t_{1}:t_{1}]
 \in\mathbf{P}^{4}\mid
 [t_{0}:t_{1}]\in\mathbf{P}^{1}\}.
 $$
 If $L\subset{V}$, then
 $$
 F(a_{0}t_{0}+b_{0}t_{1},\dots,a_{3}t_{0}+b_{3}t_{1})+t_{1}^{3}=0
 \in\mathbf{C}[t_{0},t_{1}].
 $$
 Since $F(b_{0},\dots,b_{3})+1=0$ and $F(a_{0},\dots,a_{3})=0$,
 we have $(b_{1},\dots,b_{3})\neq(0,\dots,0)$ and
 $[a_{0}:\dots:a_{3}]\neq[b_{0}:\dots:b_{3}]$, hence
 $$
 \rho(L)=\{[a_{0}t_{0}+b_{0}t_{1}:\dots:a_{3}t_{0}+b_{3}t_{1}]
 \in\mathbf{P}^{3}\mid
 [t_{0}:t_{1}]\in\mathbf{P}^{1}\}.
 $$
 is a line in $\mathbf{P}^{3}$.
 Since
 $
 F(a_{0}t_{0}+b_{0}t_{1},\dots,a_{3}t_{0}+b_{3}t_{1})=-t_{1}^{3},
 $
 the line $\rho(L)$ intersects $X$ at the point
 $[a_{0}:\dots:a_{3}]\in\mathbf{P}^{3}$ with multiplicity $3$.
\end{proof}
Let $Z=Z_{3}$ be the surface in Remark~\ref{z3}.
By Lemma~\ref{tc}, the line $\rho(L)$ represents a point of $Z$ for a
line $L$ on $V$.
Let us abuse notation by
$$
\rho:S\longrightarrow{Z};\
L\longmapsto\rho(L).
$$
We set
$$
S_{\infty}=\{L\in{\Lambda(\mathbf{P}^{4})}\mid{L\subset{V\cap{H_{4}}}}\},
$$
which is a set of $27$ points on $S$.
\begin{lemma}
 $\rho:
 S\rightarrow{Z}
 $
 is the quotient morphism by the $\Gal{(\rho)}$-action, and $S_{\infty}$
 is the set of the fixed point by the $\Gal{(\rho)}$-action on $S$.
\end{lemma}
\begin{proof}
 Let $\omega\in\mathbf{C}$ be a primitive $3$-rd root of unity.
 The automorphism
 $$
 \sigma:V\longrightarrow{V};\
 [x_{0}:\dots:x_{3}:x_{4}]\longmapsto[x_{0}:\dots:x_{3}:\omega{x_{4}}]
 $$
 is a generator of the Galois group $\Gal{(\rho)}$.
 For a line $L$ on $V$, we have $\rho(L)=\rho(\sigma(L))$, and if
 $L=\sigma(L)$, then $L$ is contained in $H_{4}$.
 Hence $S_{\infty}$ is the set of fixed points of the
 $\Gal{(\rho)}$-action on $S$.
 Let
 $$
 L'=\{[a_{0}t_{0}+b_{0}t_{1}:\dots:a_{3}t_{0}+b_{3}t_{1}]
 \in\mathbf{P}^{3}\mid
 [t_{0}:t_{1}]\in\mathbf{P}^{1}\}
 $$
 be a line in $\mathbf{P}^{3}$ which intersects $X$ at
 $[a_{0}:\dots:a_{3}]$ with multiplicity $\geq3$.
 Then there exists  $c\in\mathbf{C}$ such that
 $$
 F(a_{0}t_{0}+b_{0}t_{1},\dots,a_{3}t_{0}+b_{3}t_{1})
 =ct_{1}^{3}.
 $$
 If a line
 $$
 L=\{[a_{0}t_{0}+b_{0}t_{1}:\dots:a_{3}t_{0}+b_{3}t_{1}:
 a_{4}t_{0}+b_{4}t_{1}]
 \in\mathbf{P}^{4}\mid
 [t_{0}:t_{1}]\in\mathbf{P}^{1}\}
 $$
 is contained in $V$, then
 $$
 -(a_{4}t_{0}+b_{4}t_{1})^{3}
 =F(a_{0}t_{0}+b_{0}t_{1},\dots,a_{3}t_{0}+b_{3}t_{1})
 =ct_{1}^{3},
 $$
 hence $a_{4}=0$ and $b_{4}^{3}=-c$.
 This imply that the morphism $\rho:S\rightarrow{Z}$ is surjective, and
 the fiber at $L'\in{Z}$ is contained in a
 $\Gal{(\rho)}$-orbit.
\end{proof}
\begin{remark}\label{qs}
 Each singularity of $Z$ is isomorphic to the quotient of
 $\mathbf{C}^{2}$ by the cyclic group generated by the action
 $(a,b)\mapsto(\omega{a},\omega{b})$.
 Hence we have
 $$
 H^{i}(Z,Z\setminus{Z_{\infty}},\mathbf{Z})\simeq
 \begin{cases}
  (\mathbf{Z}/3\mathbf{Z})^{\oplus27}&\text{if $i=3$,}\\
  \mathbf{Z}^{\oplus27}&\text{if $i=4$,}\\
  0&\text{if $i\neq3,4$}.
 \end{cases}
 $$
\end{remark}
Let $\phi:Y=Y_{3}\rightarrow{X}$ be the double cover branched along its
Hessian, and let $Y_{\infty}$ be the distinguished divisor on $Y$ which
is introduced in Section~\ref{vl}.
By Remark~$\ref{qs}$, the restriction homomorphism
$
H^{2}(Z,\mathbf{Z})\rightarrow
H^{2}(Z\setminus{Z_{\infty}},\mathbf{Z})
\simeq
H^{2}(Y\setminus{Y_{\infty}},\mathbf{Z})
$
is injective with a finite cokernel, hence
$
\psi^{*}:H^{2}(Z,\mathbf{Z})\rightarrow
H^{2}(Y,\mathbf{Z})
$
is injective.
Since $H^{2}(Y,\mathbf{Z})$ is torsion free,
$H^{2}(Z,\mathbf{Z})$ is also torsion free.
The period integral
$$
H^{0}(Y,\Omega_{Y}^{2}(\log{Y_{\infty}}))
\longrightarrow
\Hom{(H_{2}(Y\setminus{Y_{\infty}},\mathbf{Z}),
\mathbf{C})};\
\omega\longmapsto
\Bigl[\gamma\mapsto\int_{\gamma}\omega\Bigr]
$$
defines Hodge structures of pure weight $2$ on
$H^{2}(Z,\mathbf{Z})$ and
$H^{2}(Z\setminus{Z_{\infty}},\mathbf{Z})$.
For $\gamma\in{H^{2}(Z\setminus{Z_{\infty}},\mathbf{Z})}$, there is a
unique $\bar{\gamma}\in{H^{2}(Z,\mathbf{Q})}$ such that the restriction
of $\bar{\gamma}$ to $H^{2}(Z\setminus{Z_{\infty}},\mathbf{Q})$ is equal
to the class of $\gamma$ in the rational cohomology group.
We define the primitive part of
$H^{2}(Z,\mathbf{Z})$ and
$H^{2}(Z\setminus{Z_{\infty}},\mathbf{Z})$
by
$$
H^{2}_{\mathrm{prim}}(Z,\mathbf{Z})
=\Ker{\bigl([\mathcal{O}_{Z}(1)]:
H^{2}(Z,\mathbf{Z})\longrightarrow
H^{4}(Z,\mathbf{Z});\
\gamma\longmapsto[\mathcal{O}_{Z}(1)]\cup{\gamma}\bigr)},
$$
$$
H^{2}_{\mathrm{prim}}(Z\setminus{Z_{\infty}},\mathbf{Z})
=\Ker{\bigl([\mathcal{O}_{Z}(1)]:
H^{2}(Z\setminus{Z_{\infty}},\mathbf{Z})\longrightarrow
H^{4}(Z,\mathbf{Q});\
\gamma\longmapsto[\mathcal{O}_{Z}(1)]\cup\bar{\gamma}\bigr)}.
$$
We define symmetric forms on
$H^{2}(Z,\mathbf{Z})$ and
$H^{2}(Z\setminus{Z_{\infty}},\mathbf{Z})$
by
$$
\langle\ ,\ \rangle_{Z}:
H^{2}(Z,\mathbf{Z})\times
H^{2}(Z,\mathbf{Z})
\longrightarrow\mathbf{Z};\
(\gamma_{1},\gamma_{2})\longmapsto
\deg{(({\gamma_{1}}\cup{\gamma_{2}})\cap[Z])},
$$
$$
\langle\ ,\ \rangle_{Z}:
H^{2}(Z\setminus{Z_{\infty}},\mathbf{Z})\times
H^{2}(Z\setminus{Z_{\infty}},\mathbf{Z})
\longrightarrow\mathbf{Q};\
(\gamma_{1},\gamma_{2})\longmapsto
\deg{((\bar{\gamma_{1}}\cup\bar{\gamma_{2}})\cap[Z])}.
$$
These symmetric forms give polarizations of Hodge structures on the
primitive part
$H^{2}_{\mathrm{prim}}(Z,\mathbf{Z})$
and
$H^{2}_{\mathrm{prim}}(Z\setminus{Z_{\infty}},\mathbf{Z})$.
\begin{proposition}\label{mp}
 The homomorphism
 $$
 H^{2}(Z\setminus{Z_{\infty}},\mathbf{Z})
 \overset{\rho^{*}}{\longrightarrow}
 H^{2}(S\setminus{S_{\infty}},\mathbf{Z})
 \simeq
 H^{2}(S,\mathbf{Z})
 $$
 induces an isomorphism
 $H^{2}(Z\setminus{Z_{\infty}},\mathbf{Z})_{\mathrm{free}}
 \simeq
 H^{2}(S,\mathbf{Z})^{\Gal{(\rho)}}$
 of Hodge structures and an isomorphism
 $$
 \bigl(H^{2}_{\mathrm{prim}}
 (Z\setminus{Z_{\infty}},\mathbf{Z})_{\mathrm{free}},\
 3\langle\ ,\ \rangle_{Z}\bigr)\simeq
 \bigl(H^{2}_{\mathrm{prim}}(S,\mathbf{Z})^{\Gal{(\rho)}},\
 \langle\ ,\ \rangle_{S}\bigr)
 $$
 of polarized Hodge structures.
\end{proposition}
\begin{proof}
 Since
 $\rho:S\setminus{S_{\infty}}\rightarrow
 Z\setminus{Z_{\infty}}$
 is a finite \'{e}tale Galois cover, we have the Cartan-Leray spectral
 sequence
 $$
 E_{2}^{p,q}
 =H^{p}(\Gal{(\rho)},H^{q}(S\setminus{S_{\infty}},\mathbf{Z}))
 \Longrightarrow
 H^{p+q}(Z\setminus{Z_{\infty}},\mathbf{Z}).
 $$
 Since the $\Gal{(\rho)}$-action on
 $H^{0}(S\setminus{S_{\infty}},\mathbf{Z})\simeq
 H^{0}(S,\mathbf{Z})\simeq\mathbf{Z}$
 is trivial, we have
 $$
 H^{p}(\Gal{(\rho)},H^{0}(S\setminus{S_{\infty}},\mathbf{Z}))
 \simeq
 \begin{cases}
  \mathbf{Z}&\text{if $p=0$,}\\
  0&\text{if $p$ is odd,}\\
  \mathbf{Z}/3\mathbf{Z}&\text{if $p\neq0$ is even.}
 \end{cases}
 $$
 Since
 $H^{1}(S\setminus{S_{\infty}},\mathbf{Z})\simeq
 H^{1}(S,\mathbf{Z})\simeq
 H^{3}(V,\mathbf{Z})$
 is a free $\mathbf{Z}$-module of rank $10$ and
 the $\Gal{(\rho)}$-action has no invariant part,
 it is regard as a free
 $\mathbf{Z}[\omega]$-module of rank $5$, where
 $\mathbf{Z}[\omega]\simeq
 \mathbf{Z}[\Gal{(\rho)}]/(\sum_{\sigma\in\Gal{(\rho)}}\sigma)$
 is the ring of Eisenstein integers \cite[(2.2)]{act}.
 Hence we have
 $$
 H^{p}(\Gal{(\rho)},H^{1}(S\setminus{S_{\infty}},\mathbf{Z}))
 \simeq
 \begin{cases}
  (\mathbf{Z}/3\mathbf{Z})^{\oplus5}&\text{if $p$ is odd,}\\
  0&\text{if $p$ is even.}
 \end{cases}
 $$
 By the spectral sequence, the homomorphism
 $$
 H^{2}(Z\setminus{Z_{\infty}},\mathbf{Z})
 \longrightarrow
 H^{0}(\Gal{(\rho)},H^{2}(S\setminus{S_{\infty}},\mathbf{Z}))
 \simeq
 H^{2}(S,\mathbf{Z})^{\Gal{(\rho)}}
 $$
 is surjective, and its kernel is of order $3^{6}$.
 Since
 $\rho^{*}\mathcal{O}_{Z}(1)=\mathcal{O}_{S}(1)$,
 we have
 $$
 \rho_{*}(([\mathcal{O}_{S}(1)]\cup\rho^{*}\bar{\gamma})\cap[S])
 =([\mathcal{O}_{Z}(1)]\cup\bar{\gamma})\cap\rho_{*}[S]
 =([\mathcal{O}_{Z}(1)]\cup\bar{\gamma})\cap3[Z]
 $$
 for $\gamma\in{H^{2}(Z\setminus{Z_{\infty}},\mathbf{Z})}$,
 hence
 $\gamma\in
 H^{2}_{\mathrm{prim}}(Z\setminus{Z_{\infty}},\mathbf{Z})$
 if and only if
 $\rho^{*}\bar{\gamma}\in
 H^{2}_{\mathrm{prim}}(S,\mathbf{Q})$.
 And we have
 $$
 \deg{((\rho^{*}\bar{\gamma_{1}}\cup\rho^{*}\bar{\gamma_{2}})\cap[S])}
 =\deg{((\bar{\gamma_{1}}\cup\bar{\gamma_{2}})\cap\rho_{*}[S])}
 =3\deg{((\bar{\gamma_{1}}\cup\bar{\gamma_{2}})\cap[Z])}
 $$
 for
 ${\gamma_{1}},{\gamma_{2}}
 \in{H^{2}(Z\setminus{Z_{\infty}},\mathbf{Z})}$.
\end{proof}
\begin{remark}\label{ld}
In the similar way, we can prove that the coinvariant part of the
$\Gal{(\rho)}$-action on $H_{2}(S,\mathbf{Z})$ is isomorphic to
$H_{2}(Z\setminus{Z_{\infty}},\mathbf{Z})$.
By the duality
$
H^{2}(Z\setminus{Z_{\infty}},\mathbf{Z})\simeq
H_{2}(Z,Z_{\infty},\mathbf{Z})\simeq
H_{2}(Z,\mathbf{Z}),
$
we have a commutative diagram
$$
\begin{CD}
 &\rho^{*}&&&&\rho_{*}&\\
 H^{2}(S\setminus{S_{\infty}},\mathbf{Z})^{\Gal{(\rho)}}&
 \overset{\simeq}{\leftarrow}&
 H^{2}(Z\setminus{Z_{\infty}},\mathbf{Z})_{\mathrm{free}}
 &\simeq&H_{2}(Z,\mathbf{Z})_{\mathrm{free}}
 &\hookleftarrow&H_{2}(S,\mathbf{Z})_{\Gal{(\rho)}}\\
 \uparrow\simeq&&\cup&&\cup&&\uparrow\simeq\\
 H^{2}(S,\mathbf{Z})^{\Gal{(\rho)}}&{\hookleftarrow}&
 H^{2}(Z,\mathbf{Z})&\simeq&H_{2}(Z\setminus{Z_{\infty}},\mathbf{Z})
 &\overset{\simeq}{\leftarrow}
 &H_{2}(S\setminus{S_{\infty}},\mathbf{Z})_{\Gal{(\rho)}}.
\end{CD}
$$
\end{remark}
\begin{remark}\label{yz}
 The restriction
 $
 H^{2}(Y,\mathbf{Z})
 \rightarrow
 H^{2}(Y\setminus{Y_{\infty}},\mathbf{Z})
 $
 induces an isomorphism
 $$
 \frac{H^{2}(Y,\mathbf{Z})}
 {\sum_{L\in{Z_{\infty}}}\mathbf{Z}L^{+}}
 \simeq
 H^{2}(Y\setminus{Y_{\infty}},\mathbf{Z})
 \simeq
 H^{2}(Z\setminus{Z_{\infty}},\mathbf{Z}),
 $$
 and the injection
 $
 \psi^{*}:H^{2}(Z,\mathbf{Z})
 \rightarrow
 H^{2}(Y,\mathbf{Z})
 $
 induces an isomorphism
 $$
 H^{2}(Z,\mathbf{Z})
 \simeq
 \Bigl(\sum_{L\in{Z_{\infty}}}\mathbf{Z}L^{+}\Bigr)^{\perp}
 \subset
 H^{2}(Y,\mathbf{Z}),
 $$
 where $\perp$ means the orthogonal complement in the unimodular lattice
 $$
 \langle\ ,\ \rangle_{Y}:
 H^{2}(Y,\mathbf{Z})\times
 H^{2}(Y,\mathbf{Z})
 \longrightarrow\mathbf{Z};\
 (\gamma_{1},\gamma_{2})\longmapsto
 \deg{(({\gamma_{1}}\cup{\gamma_{2}})\cap[Y])}.
 $$
\end{remark}
\begin{proposition}\label{tor}
 The homomorphism
 $$
 H^{2}(X,\mathbf{Z})
 \overset{\phi^{*}}{\longrightarrow}
 H^{2}(Y\setminus{Y_{\infty}},\mathbf{Z})
 {\simeq}H^{2}(Z\setminus{Z_{\infty}},\mathbf{Z})
 $$
 induces an isomorphism
 $$
 \frac{H_{\mathrm{prim}}^{2}(X,\mathbf{Z})}
 {3H_{\mathrm{prim}}^{2}(X,\mathbf{Z})}\simeq
 H^{2}
 (Z\setminus{Z_{\infty}},\mathbf{Z})_{\mathrm{tor}}
 $$
 of abelian groups.
\end{proposition}
\begin{proof}
 Since
 $
 \psi^{*}H^{2}(Z,\mathbf{Z})=
 \bigl(\sum_{L\in{Z_{\infty}}}\mathbf{Z}L^{+}\bigr)^{\perp}
 \subset{H^{2}(Y,\mathbf{Z})},
 $
 the primitive closure of the sublattice
 $\sum_{L\in{Z_{\infty}}}\mathbf{Z}L^{+}$ in $H^{2}(Y,\mathbf{Z})$ is
 $
 \bigl(\psi^{*}H^{2}(Z,\mathbf{Z})\bigr)^{\perp}
 \subset
 H^{2}(Y,\mathbf{Z}),
 $
 hence the torsion part of $H^{2}(Z\setminus{Z_{\infty}},\mathbf{Z})$ is
 $$
 H^{2}(Z\setminus{Z_{\infty}},\mathbf{Z})_{\mathrm{tor}}
 \simeq
 \Bigl(
 \frac{H^{2}(Y,\mathbf{Z})}{\sum_{L\in{Z_{\infty}}}\mathbf{Z}L^{+}}
 \Bigr)_{\mathrm{tor}}
 \simeq
 \frac{(\psi^{*}H^{2}(Z,\mathbf{Z}))^{\perp}}
 {\sum_{L\in{Z_{\infty}}}\mathbf{Z}L^{+}}.
 $$
 By the proof of Proposition~\ref{mp}, this is an abelian group of order
 $3^{6}$, hence the sublattice
 $(\psi^{*}H^{2}(Z,\mathbf{Z}))^{\perp}
 \subset
 H^{2}(Y,\mathbf{Z})$
 is of rank $27$ and
 $$
 \det{(\psi^{*}H^{2}(Z,\mathbf{Z}))^{\perp}}=(3^{6})^{-2}\cdot
 \det{\bigl(\sum_{L\in{Z_{\infty}}}\mathbf{Z}L^{+}\bigr)}
 =-3^{15}.
 $$
 Since $H_{\mathrm{prim}}^{2}(X,\mathbf{Z})$ is generated by the
 difference of two lines on $X$,
 by Proposition~\ref{rel}, we have
 $
 3\phi^{*}H_{\mathrm{prim}}^{2}(X,\mathbf{Z})\subset
 \sum_{L\in{Z_{\infty}}}\mathbf{Z}L^{+}
 $
 and
 $
 \phi^{*}H_{\mathrm{prim}}^{2}(X,\mathbf{Z})\subset
 (\psi^{*}H^{2}(Z,\mathbf{Z}))^{\perp}.
 $
 By Remark~$\ref{int}$, we can directly compute the determinant of
 the sublattice
 $
 \phi^{*}H_{\mathrm{prim}}^{2}(X,\mathbf{Z})
 +\sum_{L\in{Z_{\infty}}}\mathbf{Z}L^{+}
 \subset
 H^{2}(Y,\mathbf{Z}),
 $
 that is
 $
 \det{\bigl(\phi^{*}H_{\mathrm{prim}}^{2}(X,\mathbf{Z})
 +\sum_{L\in{Z_{\infty}}}\mathbf{Z}L^{+}\bigr)}
 =-3^{15}.
 $
 Hence we have
 $$
 \phi^{*}H_{\mathrm{prim}}^{2}(X,\mathbf{Z})
 +\sum_{L\in{Z_{\infty}}}\mathbf{Z}L^{+}
 =\psi^{*}H^{2}(Z,\mathbf{Z})^{\perp}.
 $$
 This implies that the homomorphism
 $$
 (\mathbf{Z}/3\mathbf{Z})^{\oplus6}\simeq
 \frac{H_{\mathrm{prim}}^{2}(X,\mathbf{Z})}
 {3H_{\mathrm{prim}}^{2}(X,\mathbf{Z})}
 \longrightarrow
 \frac{(\psi^{*}H^{2}(Z,\mathbf{Z}))^{\perp}}
 {\sum_{L\in{Z_{\infty}}}\mathbf{Z}L^{+}}
 $$
 is surjective.
 Since the order of these groups are both equal to $3^{6}$, it is an
 isomorphism.
\end{proof}
By proposition~$\ref{tor}$ and Remark~\ref{yz},
we have the isomorphism
$$
\frac{H^{2}(Y,\mathbf{Z})}
{\phi^{*}H^{2}_{\mathrm{prim}}(X,\mathbf{Z})
+\sum_{L\in{Z_{\infty}}}\mathbf{Z}L^{+}}
\simeq
H^{2}(Z\setminus{Z_{\infty}},\mathbf{Z})_{\mathrm{free}}.
$$
We denote by
$\bigl(\frac{H^{2}(Y,\mathbf{Z})}
{\phi^{*}H^{2}_{\mathrm{prim}}(X,\mathbf{Z})
+\sum_{L\in{Z_{\infty}}}\mathbf{Z}L^{+}}\bigr)_{0}$
the subspace of
$\frac{H^{2}(Y,\mathbf{Z})}
{\phi^{*}H^{2}_{\mathrm{prim}}(X,\mathbf{Z})
+\sum_{L\in{Z_{\infty}}}\mathbf{Z}L^{+}}$
orthogonal to
$[\psi^{*}\mathcal{O}_{Z}(1)]\in{H^{2}(Y,\mathbf{Z})}$.
We denote by
$\bigl(\bigwedge^{2}H^{3}(V,\mathbf{Z})\bigr)_{0}$
the kernel of the homomorphism
$$
\bigwedge^{2}H^{3}(V,\mathbf{Z})
\longrightarrow
\mathbf{Z};\
\alpha_{1}\wedge\alpha_{2}\longmapsto
\deg{((\alpha_{1}\cup\alpha_{2})\cap[V])},
$$
and denote by $H^{3}(V,\mathbf{Z})(1)$ the Hodge structure of weight $1$
which is defined from the Hodge structure $H^{3}(V,\mathbf{Z})$ by the
shift of the weight.
\begin{theorem}\label{mt}
 There is a natural injective homomorphism
 $$
 \Bigl(\bigwedge^{2}H^{3}(V,\mathbf{Z})(1)\Bigr)^{\Gal{(\rho)}}
 \longrightarrow
 \frac{H^{2}(Y,\mathbf{Z})}
 {\phi^{*}H^{2}_{\mathrm{prim}}(X,\mathbf{Z})
 +\sum_{L\in{Z_{\infty}}}\mathbf{Z}L^{+}}
 $$
 with the cokernel $\mathbf{Z}/2\mathbf{Z}$, which induces an
 isomorphism
 $$
 \Bigl(\bigwedge^{2}H^{3}(V,\mathbf{Z})(1)\Bigr)_{0}^{\Gal{(\rho)}}
 \simeq
 \Bigl(\frac{H^{2}(Y,\mathbf{Z})}
 {\phi^{*}H^{2}_{\mathrm{prim}}(X,\mathbf{Z})
 +\sum_{L\in{Z_{\infty}}}\mathbf{Z}L^{+}}\Bigr)_{0}
 $$
 of Hodge structures.
\end{theorem}
\begin{proof}
 By Theorem~$\ref{aj}$, \cite[Lemma~9.13 and (10.14)]{cg},
 Proposition~$\ref{mp}$, Remark~\ref{yz} and
 Proposition~\ref{tor}, we have the following sequence of homomorphisms
 of Hodge structures;
 $$
 \begin{matrix}
  \bigl(\bigwedge^{2}H^{3}(V,\mathbf{Z})(1)\bigr)^{\Gal{(\rho)}}
  &\overset{\sim}{\rightarrow}&
  \bigl(\bigwedge^{2}H^{1}(S,\mathbf{Z})\bigr)^{\Gal{(\rho)}}
  &\overset{\sim}{\leftarrow}&
  \bigl(\bigwedge^{2}H^{1}(A,\mathbf{Z})\bigr)^{\Gal{(\rho)}}\\
  &&\cap&&\downarrow\simeq\\
  H^{2}(S\setminus{S_{\infty}},\mathbf{Z})^{\Gal{(\rho)}}
  &\overset{\sim}{\leftarrow}&
  H^{2}(S,\mathbf{Z})^{\Gal{(\rho)}}
  &\supset&
  H^{2}(A,\mathbf{Z})^{\Gal{(\rho)}}\\
  \simeq\uparrow&&&&\\
  H^{2}(Z\setminus{Z_{\infty}},\mathbf{Z})_{\mathrm{free}}
  &\overset{\sim}{\rightarrow}&
  H^{2}(Y\setminus{Y_{\infty}},\mathbf{Z})_{\mathrm{free}}
  &\overset{\sim}{\leftarrow}&
  \frac{H^{2}(Y,\mathbf{Z})}
  {\phi^{*}H^{2}_{\mathrm{prim}}(X,\mathbf{Z})
  +\sum_{L\in{Z_{\infty}}}\mathbf{Z}L^{+}}.
 \end{matrix}
 $$
 Since
 $\bigl(\bigwedge^{2}H^{3}(V,\mathbf{Z})(1)\bigr)_{0}$
 corresponds to
 $H_{\mathrm{prim}}^{2}(A,\mathbf{Z})$,
 and
 $\bigl(\frac{H^{2}(Y,\mathbf{Z})}
 {\phi^{*}H^{2}_{\mathrm{prim}}(X,\mathbf{Z})
 +\sum_{L\in{Z_{\infty}}}\mathbf{Z}L^{+}}\bigr)_{0}$
 corresponds to
 $H_{\mathrm{prim}}^{2}(S,\mathbf{Z})^{\Gal{(\rho)}}$,
 by Proposition~$\ref{ht}$ we have the isomorphism
 $$
 \Bigl(\bigwedge^{2}H^{3}(V,\mathbf{Z})(1)\Bigr)_{0}^{\Gal{(\rho)}}
 \simeq
 \Bigl(\frac{H^{2}(Y,\mathbf{Z})}
 {\phi^{*}H^{2}_{\mathrm{prim}}(X,\mathbf{Z})
 +\sum_{L\in{Z_{\infty}}}\mathbf{Z}L^{+}}\Bigr)_{0}.
 $$
\end{proof}
We denote by $A_{i}$ the positive definite root lattice of type $A_{i}$,
and by $\mathbf{1}$ the trivial lattice of rank $1$.
\begin{proposition}\label{lat}
 There are isomorphisms of lattices;
 $$
 \bigl(H^{2}(Z\setminus{Z_{\infty}},\mathbf{Z})_{\mathrm{free}},
 \langle\ ,\ \rangle_{Z}\bigr)\simeq
 \bigl({{\frac{1}{3}}\cdot\mathbf{1}}\bigr)\oplus
 \bigl(-{\frac{1}{3}}\cdot\mathbf{1}\bigr)^{\oplus4}
 \oplus
 \bigl(\frac{1}{3}{\cdot}A_{2}\bigr)^{\oplus4}
 \oplus{\bigl(-\frac{1}{3}{\cdot}A_{2}\bigr)}^{\oplus6},
 $$
 $$
 \bigl(H^{2}_{\mathrm{prim}}
 (Z\setminus{Z_{\infty}},\mathbf{Z})_{\mathrm{free}},\
 \langle\ ,\ \rangle_{Z}\bigr)\simeq
 {\bigl(-\frac{1}{3}{\cdot}A_{4}\bigr)}\oplus
 \bigl(\frac{1}{3}{\cdot}A_{2}\bigr)^{\oplus4}
 \oplus{\bigl(-\frac{1}{3}{\cdot}A_{2}\bigr)}^{\oplus6},
 $$
 $$
 \bigl(H^{2}(Z,\mathbf{Z}),\langle\ ,\ \rangle_{Z}\bigr)\simeq
 (3\cdot{\mathbf{1}})\oplus(-3\cdot\mathbf{1})^{\oplus4}\oplus
 A_{2}^{\oplus4}\oplus{(-A_{2})}^{\oplus6},
 $$
 $$
 \bigl(H_{\mathrm{prim}}^{2}(Z,\mathbf{Z}),
 \langle\ ,\ \rangle_{Z}\bigr)\simeq
 {(-3{\cdot}A_{4})}\oplus
 A_{2}^{\oplus4}\oplus{(-A_{2})}^{\oplus6}.
 $$
\end{proposition}
We define an alternating form on $H^{1}(A,\mathbf{Z})$ by
$$
\langle\ ,\ \rangle_{A}:
H^{1}(A,\mathbf{Z})\times
H^{1}(A,\mathbf{Z})\longrightarrow
\mathbf{Z};\
(\alpha_{1},\alpha_{2})\longmapsto
\deg{\Bigl(\Bigl(\frac{\theta^{\cup4}}{4!}
\cup\alpha_{1}\cup\alpha_{2}\Bigr)\cap[A]\Bigr)}.
$$
\begin{lemma}[\cite{act}\ (2.7)]\label{sb}
 There is a basis $(v_{0},\dots,v_{4})$ of the
 $\mathbf{Z}[\omega]$-module $H^{1}(A,\mathbf{Z})$ such that
 $$
 \Bigl(
  \langle{v_{i}},v_{j}\rangle_{A}
 \Bigr)
 _{0\leq{i,j}\leq{4}}
 =
 \begin{pmatrix}
  0&0&0&0&0\\
  0&0&0&0&0\\
  0&0&0&0&0\\
  0&0&0&0&0\\
  0&0&0&0&0
 \end{pmatrix},\quad
 \Bigl(\langle{v_{i}},\omega{v_{j}}\rangle_{A}
 \Bigr)
 _{0\leq{i,j}\leq{4}}
 =
 \begin{pmatrix}
  -1&0&0&0&0\\
  0&1&0&0&0\\
  0&0&1&0&0\\
  0&0&0&1&0\\
  0&0&0&0&1
 \end{pmatrix}.
 $$
\end{lemma}
\begin{proof}[Proof of Proposition~\ref{lat}]
Using the basis in Lemma~\ref{sb}, the class of the principal
polarization is
$$
\theta=-v_{0}\cup\omega{v_{0}}+
\sum_{i=1}^{4}v_{i}\cup\omega{v_{i}}
\in
H^{2}(A,\mathbf{Z}).
$$
We set
$
\tau=\frac{1}{2}\theta\in
H^{2}(A,\mathbf{Q}).
$
Then $\tau$ corresponds to the class of an incidence divisor on $S$, and
$H^{2}(S,\mathbf{Z})$
is identified with the sublattice in $H^{2}(A,\mathbf{Q})$ generated by
$\tau$ and $H^{2}(A,\mathbf{Z})$.
We define sublattices in $H^{2}(A,\mathbf{Q})$ by
$$
U_{0}=\bigoplus_{i=0}^{4}
\mathbf{Z}{v_{i}}{\cup}\omega{v_{i}}
\subset
H^{2}(A,\mathbf{Z}),
$$
$$
\tilde{U}_{0}=
\mathbf{Z}\tau+U_{0}
=
\mathbf{Z}\tau\oplus
\bigoplus_{i=1}^{4}
\mathbf{Z}{v_{i}}{\cup}\omega{v_{i}}
\subset
H^{2}(A,\mathbf{Q}),
$$
$$
U'_{0}=
\mathbf{Z}({v_{0}}{\cup}\omega{v_{0}}
+{v_{1}}{\cup}\omega{v_{1}})
\oplus
\bigoplus_{i=1}^{3}
\mathbf{Z}({v_{i+1}}{\cup}\omega{v_{i+1}}
-{v_{i}}{\cup}\omega{v_{i}})
\subset
H^{2}(A,\mathbf{Z})
$$
and
$$
U_{i,j}=\mathbf{Z}v_{i}{\cup}v_{j}
\oplus\mathbf{Z}{v_{i}}{\cup}\omega{v_{j}}
\oplus\mathbf{Z}\omega{v_{i}}{\cup}{v_{j}}
\oplus\mathbf{Z}\omega{v_{i}}{\cup}\omega{v_{j}}
\subset
H^{2}(A,\mathbf{Z})
$$
for $0\leq{i<j}\leq4$.
Then we have orthogonal decompositions of lattices
$$
H^{2}(S,\mathbf{Z})=\tilde{U}_{0}
\oplus\bigoplus_{0\leq{i<j}\leq{4}}U_{i,j},
$$
$$
H_{\mathrm{prim}}^{2}(S,\mathbf{Z})
\simeq
H_{\mathrm{prim}}^{2}(A,\mathbf{Z})
=U'_{0}\oplus\bigoplus_{0\leq{i<j}\leq{4}}U_{i,j},
$$
which are compatible with the $\Gal{(\rho)}$-action.
The $\Gal{(\rho)}$-action on
$\tilde{U}_{0}\simeq\mathbf{1}\oplus(-\mathbf{1})^{\oplus4}$
and
$U'_{0}\simeq(-A_{4})$ are trivial, and the invariant parts of the
$\Gal{(\rho)}$-action on $U_{i,j}$ are
$$
U_{0,j}^{\Gal{(\rho)}}=
\mathbf{Z}({v_{0}}\cup{v_{j}}
+\omega{v_{0}}\cup\omega{v_{j}}
+\omega{v_{0}}\cup{v_{j}})
\oplus\mathbf{Z}({v_{0}}\cup{v_{j}}
+\omega{v_{0}}\cup\omega{v_{j}}
+{v_{0}}\cup\omega{v_{j}})
\simeq
A_{2}
$$
for $1\leq{j}\leq4$, and
$$
U_{i,j}^{\Gal{(\rho)}}=
\mathbf{Z}({v_{i}}\cup{v_{j}}
+\omega{v_{i}}\cup\omega{v_{j}}
+\omega{v_{i}}\cup{v_{j}})
\oplus\mathbf{Z}({v_{i}}\cup{v_{j}}
+\omega{v_{i}}\cup\omega{v_{j}}
+{v_{i}}\cup\omega{v_{j}})
\simeq
(-A_{2})
$$
for $1\leq{i<j}\leq{4}$.
Hence we have
$$
\bigl(H^{2}(S,\mathbf{Z})^{\Gal{(\rho)}},\langle\ ,\ \rangle_{S}\bigr)
\simeq
{\mathbf{1}}\oplus(-\mathbf{1})^{\oplus4}\oplus
A_{2}^{\oplus4}\oplus{(-A_{2})}^{\oplus6}
$$
and
$$
\bigl(H_{\mathrm{prim}}^{2}(S,\mathbf{Z})^{\Gal{(\rho)}},
\langle\ ,\ \rangle_{S}\bigr)\simeq
(-A_{4})\oplus
A_{2}^{\oplus4}\oplus{(-A_{2})}^{\oplus6}.
$$
By Proposition~\ref{mp}, we have the results for lattices
$H^{2}(Z\setminus{Z_{\infty}},\mathbf{Z})_{\mathrm{free}}$
and
$H_{\mathrm{prim}}^{2}
(Z\setminus{Z_{\infty}},\mathbf{Z})_{\mathrm{free}}$.
In the similar way, the statements for lattices
$H^{2}(Z,\mathbf{Z})$ and
$H_{\mathrm{prim}}^{2}(Z,\mathbf{Z})$
can be proved.
\end{proof}
\begin{proposition}\label{pr}
 $$
 \phi^{*}H_{\mathrm{prim}}^{2}(X,\mathbf{Z})
 +\sum_{L\in{Z_{\infty}}}\mathbf{Z}L^{+}
 =\bigl(\psi^{*}H^{2}(Z,\mathbf{Z})\bigr)^{\perp},
 $$
 $$
 \phi^{*}H^{2}(X,\mathbf{Z})
 +\sum_{L\in{Z_{\infty}}}\mathbf{Z}L^{+}
 =\bigl(\psi^{*}H_{\mathrm{prim}}^{2}(Z,\mathbf{Z})\bigr)^{\perp}.
 $$
\end{proposition}
\begin{proof}
 The first equality has been proved in the proof of
 Proposition~\ref{tor}.
 Since
 $$
 \psi^{*}H^{2}_{\mathrm{prim}}(Z,\mathbf{Z})
 =
 \bigl(\phi^{*}H^{2}(X,\mathbf{Z})
 +\sum_{L\in{Z_{\infty}}}\mathbf{Z}L^{+}\bigr)^{\perp},
 $$
 we have
 $$
 \phi^{*}H^{2}(X,\mathbf{Z})
 +\sum_{L\in{Z_{\infty}}}\mathbf{Z}L^{+}
 \subset
 (\psi^{*}H^{2}_{\mathrm{prim}}(Z,\mathbf{Z}))^{\perp},
 $$
 which are sublattices of rank $28$.
 We compute the determinant of these lattices.
 By Proposition~\ref{lat}, we have
 $$
 \det{H^{2}_{\mathrm{prim}}(Z,\mathbf{Z})}
 =3^{4}\cdot\det{(-A_{4})}
 \cdot(\det{A_{2}})^{4}\cdot(\det{(-A_{2})})^{6}
 =3^{14}\cdot5,
 $$
 hence
 $\det{(\psi^{*}H^{2}_{\mathrm{prim}}(Z,\mathbf{Z}))^{\perp}}
 =-3^{14}\cdot5$.
 On the other hand, by Remark~$\ref{int}$, we can directly compute
 the determinant of the sublattice
 $
 \phi^{*}H^{2}(X,\mathbf{Z})
 +\sum_{L\in{Z_{\infty}}}\mathbf{Z}L^{+}
 \subset
 H^{2}(Y,\mathbf{Z}),
 $
 that is
 $$
 \det{\Bigl(\phi^{*}H^{2}(X,\mathbf{Z})
 +\sum_{L\in{Z_{\infty}}}\mathbf{Z}L^{+}\Bigr)}
 =-3^{14}\cdot5,
 $$
 hence we have the second equality.
\end{proof}
\section{N\'{e}ron-Severi lattice}\label{nsl}
The N\'{e}ron-Severi group $\NS{(Y)}$ of the surface $Y$ is the subgroup
of $H^{2}(Y,\mathbf{Z})$ generated by algebraic cycles.
Since $H^{2}(X,\mathbf{Z})$ is generated by algebraic cycles,
$$
\NS{(Y)}_{0}=
\phi^{*}H^{2}(X,\mathbf{Z})
+\sum_{L\in{Z_{\infty}}}\mathbf{Z}L^{+}
\subset
H^{2}(Y,\mathbf{Z})
$$
is contained in $\NS{(Y)}$.
By the proof of Proposition~\ref{pr}, $\NS_{0}{(Y)}$ is a sublattice of
rank $28$ with the determinant $-3^{14}\cdot5$.
If there are no Eckardt points on $X$, then
$\NS{(Y)}_{0}=\sum_{L\in{Z_{\infty}}}(\mathbf{Z}L^{+}+\mathbf{Z}L^{-})$.
\begin{theorem}\label{nsg}
 $\NS{(Y)}=\NS{(Y)}_{0}$
 for a generic cubic surface $X$.
\end{theorem}
The idea of the proof is based on the theory of infinitesimal variations
of Hodge structure \cite[Section~3]{cggh}.
Let
$\mathcal{M}\subset\mathbf{P}
(H^{0}(\mathbf{P}^{3},\mathcal{O}_{\mathbf{P}^{3}}(3))^{\vee})$
be the space of smooth cubic surfaces, and let
$\mathcal{Y}\rightarrow\mathcal{M}$
be the family of the surface $Y$.
We define a homomorphism by
$$
\epsilon:
H^{1}(Y,\Omega_{Y}^{1})\longrightarrow
\Hom{(T_{\mathcal{M}}([F]),H^{2}(Y,\mathcal{O}_{Y}))};\
\omega\longmapsto
\bigl[\xi\mapsto
c(\kappa(\xi)\cup\omega)
\bigr],
$$
where $T_{\mathcal{M}}([F])$ is the tangent space of $\mathcal{M}$
at $[F]\in{\mathcal{M}}$, $Y$ is the fiber of
$\mathcal{Y}\rightarrow\mathcal{M}$ at $[F]\in{\mathcal{M}}$,
$\kappa(\xi)\in{H^{1}(Y,T_{Y})}$ is the Kodaira-Spencer class of
$\xi{\in}T_{\mathcal{M}}([F])$, and
$$
c:H^{2}(Y,T_{Y}\otimes\Omega_{Y}^{1})
\longrightarrow
H^{2}(Y,\mathcal{O}_{Y})
$$
is the contraction homomorphism.
We remark that
$\mathbf{C}\otimes_{\mathbf{Z}}\NS{(Y)}$ is
isomorphic to the kernel of $\epsilon$ for a generic $[F]\in\mathcal{M}$.
\begin{proposition}\label{ivd}
 The homomorphism
 $\epsilon:
 H^{1}(Y,\Omega_{Y}^{1})\rightarrow
 \Hom{(T_{\mathcal{M}}([F]),H^{2}(Y,\mathcal{O}_{Y}))}$
 is of rank $16$.
\end{proposition}
The computation of the infinitesimal variations of Hodge structure for
$Y$ is given in Section~$\ref{jr}$, and Proposition~$\ref{ivd}$ will be
proved there.
\begin{proof}[Proof of Theorem~$\ref{nsg}$]
 By Proposition~\ref{pr}, $\frac{\NS{(Y)}}{\NS_{0}{(Y)}}$ has no torsion
 element, and by Proposition~$\ref{ivd}$,
 the rank of $\NS{(Y)}$ is $28$ for
 a generic cubic surface $X$.
 Hence we have $\frac{\NS{(Y)}}{\NS_{0}{(Y)}}=0$ for a generic $X$.
\end{proof}
Next we study the surface $Y$ for the Fermat cubic surface $X$.
Let $X\subset\mathbf{P}^{3}$ be the cubic surface defined by
$F=x_{0}^{3}+\dots+x_{3}^{3}$.
Then the triple Galois cover $V$ of $\mathbf{P}^{3}$ branched along $X$
is the Fermat cubic $3$-fold defined by
$\tilde{F}=x_{0}^{3}+\dots+x_{3}^{3}+x_{4}^{3}$.
We set a point $e_{i,j}^{\alpha}$ on $V$ by
$$
e_{i,j}^{\alpha}=
\bigl\{[x_0:\cdots:x_{4}]\in\mathbf{P}^{4}\mid
x_{i}+\alpha{x_{j}=0},\
x_{k}=0\ \text{for}\ k\in\{0,1,\dots,4\}\setminus\{i,j\}
\bigr\}
$$
for $0\leq{i<j}\leq{4}$ and $\alpha\in\mathbf{C}$ with
$\alpha^{3}=1$.
The point $e_{i,j}^{\alpha}$ corresponds to an elliptic curve
$\mathcal{E}_{i,j}^{\alpha}$ on the Fano surface $S$ of lines on
$V\subset\mathbf{P}^{4}$ by
$$
\mathcal{E}_{i,j}^{\alpha}=
\{L\in{S}\mid{e_{i,j}^{\alpha}}\in{L}\}.
$$
\begin{theorem}[Roulleau \cite{r}, Theorem~3.13]\label{rou}
 For the Fermat cubic $3$-fold $V$, the N\'{e}ron-Severi lattice
 $\NS{(S)}$ is of rank $25$ with the determinant $3^{18}$, and
 $$
 \NS{(S)}=\mathbf{Z}\tau
 +\sum_{0\leq{i<j}\leq{4}}(\mathbf{Z}\mathcal{E}_{i,j}^{1}
 +\mathbf{Z}\mathcal{E}_{i,j}^{\omega}
 +\mathbf{Z}\mathcal{E}_{i,j}^{\omega^{2}}),
 $$
 where $\tau$ is the class of an incidence divisor.
\end{theorem}
By using Theorem~$\ref{rou}$, we compute the N\'{e}ron-Severi lattice
$\NS{(Y)}$ for the Fermat cubic surface $X$.
The branch divisor $B$ of the double cover
$\phi:Y\rightarrow{X}$ is the sum of the elliptic curves
$$
B_{k}=\{[x_{0}:\dots:x_{3}]\in{X}\mid{x_{k}=0}\}
$$
for $0\leq{k}\leq{3}$, because the Hessian of $F$ is
$6^{4}x_{0}x_{1}x_{2}x_{3}$.
Let $D_{k}$ be the irreducible component of the ramification divisor $R$
of $\phi:Y\rightarrow{X}$ which corresponds to $B_{k}$,
and let $E_{i,j}^{\alpha}$ be the irreducible component of the
ramification divisor $R$ which corresponds to the Eckardt point
$$
\rho(e_{i,j}^{\alpha})=
\bigl\{[x_0:\cdots:x_{3}]\in\mathbf{P}^{3}\mid
x_{i}+\alpha{x_{j}=0},\
x_{k}=0\ \text{for}\ k\in\{0,1,2,3\}\setminus\{i,j\}
\bigr\}
$$
for $0\leq{i<j}\leq{3}$ and $\alpha\in\mathbf{C}$ with
$\alpha^{3}=1$.
We remark that $D_{k}$ is an elliptic curve, and the irreducible
decomposition of the ramification divisor is
$$
R=\sum_{k=0}^{3}D_{k}+\sum_{0\leq{i<j}\leq3}
(E_{i,j}^{1}+E_{i,j}^{\omega}+E_{i,j}^{\omega^{2}}).
$$
\begin{remark}\label{int2}
 For a line $L$ on the Fermat cubic surface $X$ and an Eckardt point $e$
 on $X$, the intersection numbers on $Y$ are computed by
 \begin{align*}
  &(D_{k}.D_{l})=
  \begin{cases}
   0&\text{if $k\neq{l}$},\\
   -3&\text{if $k=l$},
  \end{cases}\\
  &(D_{k}.L^{+})=(D_{k}.L^{-})=0,\\
  &(D_{k}.\phi^{-1}(e))=
  \begin{cases}
   0&\text{if $e\notin{B_{k}}$},\\
   1&\text{if $e\in{B_{k}}$}.
  \end{cases}
 \end{align*}
\end{remark}
\begin{lemma}\label{chi}
 There is an isomorphism
 $$
 \chi:
 \NS{(S)}^{\Gal{(\rho)}}\overset{\sim}{\longrightarrow}
 \frac{\NS{(Y)}}
 {\phi^{*}H^{2}_{\mathrm{prim}}(X,\mathbf{Z})+
 \sum_{L\in{Z_{\infty}}}\mathbf{Z}L^{+}}
 $$
 such that
 $$
 \begin{cases}
  \chi(\tau)=\pi(\phi^{*}L)
  &\text{for a line $L$ on X},\\
  \chi(\mathcal{E}_{i,j}^{\alpha})=\pi({E}_{i,j}^{\alpha})
  &\text{for $0\leq{i}\leq{j}\leq3$ and $\alpha^{3}=1$},\\
  \chi(\mathcal{E}_{k,4}^{1}+
  \mathcal{E}_{k,4}^{\omega}+\mathcal{E}_{k,4}^{\omega^{2}})
  =\pi(D_{k})
  &\text{for $0\leq{k}\leq3$},
 \end{cases}
 $$
 where $\pi$ denotes the natural surjective homomorphism
 $$
 \pi:\NS{(Y)}\longrightarrow
 \frac{\NS{(Y)}}
 {\phi^{*}H^{2}_{\mathrm{prim}}(X,\mathbf{Z})+
 \sum_{L\in{Z_{\infty}}}\mathbf{Z}L^{+}}.
 $$
\end{lemma}
\begin{proof}
 By Proposition~\ref{mp}, Remark~\ref{yz} and Proposition~\ref{tor}, we
 have the isomorphism of Hodge structures
 $$
 H^{2}(S,\mathbf{Z})^{\Gal{(\rho)}}\simeq
 \frac{H^{2}(Y,\mathbf{Z})}
 {\phi^{*}H^{2}_{\mathrm{prim}}(X,\mathbf{Z})+
 \sum_{L\in{Z_{\infty}}}\mathbf{Z}L^{+}},
 $$
 and this induces the isomorphism
 $$
 \chi:
 \NS{(S)}^{\Gal{(\rho)}}\overset{\sim}{\longrightarrow}
 \frac{\NS{(Y)}}
 {\phi^{*}H^{2}_{\mathrm{prim}}(X,\mathbf{Z})+
 \sum_{L\in{Z_{\infty}}}\mathbf{Z}L^{+}}.
 $$
 Since
 $3\tau=[\mathcal{O}_{S}(1)]=\rho^{*}[\mathcal{O}_{Z}(1)]$
 by \cite[\textsection{10}]{cg},
 and
 $\pi(\psi^{*}[\mathcal{O}_{Z}(1)])=\pi(3\phi^{*}L)$
 by Proposition~$\ref{rel}$,
 we have $\chi(3\tau)=\pi(3\phi^{*}L)$.
 Since
 $\frac{\NS{(Y)}}
 {\phi^{*}H^{2}_{\mathrm{prim}}(X,\mathbf{Z})+
 \sum_{L\in{Z_{\infty}}}\mathbf{Z}L^{+}}$
 is torsion free, we have $\chi(\tau)=\pi(\phi^{*}L)$.
 The triple cover
 $\rho:S\rightarrow{Z}$ induces a triple cover
 $\mathcal{E}_{i,j}^{\alpha}\rightarrow\psi({E}_{i,j}^{\alpha})$
 for $0\leq{i}\leq{j}\leq3$, and an isomorphism
 $\mathcal{E}_{k,4}^{\alpha}\overset{\sim}{\rightarrow}\psi(D_{k})$
 for $0\leq{k}\leq3$.
 These imply that
 $\chi(\mathcal{E}_{i,j}^{\alpha})=\pi({E}_{i,j}^{\alpha})$ and
 $\chi(\mathcal{E}_{k,4}^{1}+
 \mathcal{E}_{k,4}^{\omega}+\mathcal{E}_{k,4}^{\omega^{2}})
 =\pi(D_{k})$.
\end{proof}
\begin{theorem}\label{nsf}
 For the Fermat cubic surface $X$, the N\'{e}ron-Severi lattice
 $\NS{(Y)}$ is of rank $44$ with the determinant $-3^{12}$, and
 $$
 \NS{(Y)}=\sum_{L\in{Z_{\infty}}}(\mathbf{Z}L^{+}+\mathbf{Z}L^{-})
 +\sum_{0\leq{i<j}\leq{3}}
 (\mathbf{Z}E_{i,j}^{1}+\mathbf{Z}E_{i,j}^{\omega}
 +\mathbf{Z}E_{i,j}^{\omega^{2}})
 +\sum_{0\leq{k}\leq{3}}\mathbf{Z}D_{k}.
 $$
\end{theorem}
\begin{proof}
 By Theorem~$\ref{rou}$, we have
 $$
 \NS{(S)}^{\Gal{(\rho)}}
 =\mathbf{Z}\tau
 +\sum_{0\leq{i<j}\leq{3}}(\mathbf{Z}\mathcal{E}_{i,j}^{1}
 +\mathbf{Z}\mathcal{E}_{i,j}^{\omega}
 +\mathbf{Z}\mathcal{E}_{i,j}^{\omega^{2}})
 +\sum_{0\leq{k}\leq3}\mathbf{Z}(\mathcal{E}_{k,4}^{1}+
 \mathcal{E}_{k,4}^{\omega}+\mathcal{E}_{k,4}^{\omega^{2}}).
 $$
 By Lemma~$\ref{chi}$, we have
 $$
 \NS{(Y)}=\phi^{*}\NS{(X)}
 +\sum_{L\in{Z_{\infty}}}\mathbf{Z}L^{+}
 +\sum_{0\leq{i<j}\leq{3}}
 (\mathbf{Z}E_{i,j}^{1}+\mathbf{Z}E_{i,j}^{\omega}
 +\mathbf{Z}E_{i,j}^{\omega^{2}})
 +\sum_{0\leq{k}\leq{3}}\mathbf{Z}D_{k},
 $$
 and by Remark~$\ref{int}$ and Remark~$\ref{int2}$, we can directly
 compute the determinant of the lattice.
\end{proof}
\begin{remark}
 The sublattice
 $$
 \sum_{L\in{Z_{\infty}}}(\mathbf{Z}L^{+}+\mathbf{Z}L^{-})
 +\sum_{0\leq{i<j}\leq{3}}
 (\mathbf{Z}E_{i,j}^{1}+\mathbf{Z}E_{i,j}^{\omega}
 +\mathbf{Z}E_{i,j}^{\omega^{2}})
 $$
 is of rank $44$ with the determinant $-2^{2}\cdot3^{12}$,
 hence it is a sublattice of index $2$ in $\NS{(Y)}$.
\end{remark}
\section{Infinitesimal variations of Hodge structure}\label{jr}
In this section, we compute the infinitesimal variations of Hodge
structure for the surface $Y\subset\Gamma(\mathbf{P}^{3})$, and we prove
Proposition~$\ref{ivd}$.
The method is introduced in \cite{i2} as a theory of Jacobian rings.
Let $Y=Y_{3}\subset{Y_{2}}\subset{Y_{1}}\subset\Gamma(\mathbf{P}^{3})$
be the varieties defined in Section~$\ref{vl}$.
Let
$$
\begin{matrix}
 \mathcal{Y}_{3}&\subset&\mathcal{M}\times\Gamma(\mathbf{P}^{3})\\
 \downarrow&\swarrow&\\
 \mathcal{M}
\end{matrix}
$$
be the family of the surface $Y_{3}$.
Let
$$
\kappa:
T_{\mathcal{M}}([F])\longrightarrow
H^{1}(Y_{3},T_{Y_{3}}),
$$
be the Kodaira-Spencer map.
By the duality, Proposition~$\ref{ivd}$ is a corollary of the
following proposition.
\begin{proposition}\label{iv}
 The homomorphism
 $$
 T_{\mathcal{M}}([F])\otimes
 H^{0}(Y_{3},\Omega_{Y_{3}}^{2})
 \longrightarrow
 H^{1}(Y_{3},\Omega_{Y_{3}}^{1});\
 \xi\otimes{\omega}\longmapsto
 c(\kappa(\xi)\cup\omega)
 $$
 is of rank $16$, where $c$ is the contraction homomorphism
 $$
 c:H^{1}(Y_{3},T_{Y_{3}}\otimes\Omega_{Y_{3}}^{2})
 \overset{\sim}{\longrightarrow}
 H^{1}(Y_{3},\Omega_{Y_{3}}^{1}).
 $$
\end{proposition}
Let $\mathcal{S}_{\mathbf{P}^{3}}$ be the kernel of the homomorphism
$
\mathcal{O}_{\mathbf{P}^{3}}\otimes{V}
\rightarrow\mathcal{Q}_{\mathbf{P}^{3}}
\simeq\mathcal{O}_{\mathbf{P}^{3}}(1),
$
where $V=H^{0}(\mathbf{P}^{3},\mathcal{O}_{\mathbf{P}^{3}}(1))$.
Let $\mathcal{S}_{\Lambda(\mathbf{P}^{3})}$ be the kernel of the
homomorphism
$
\mathcal{O}_{\Lambda(\mathbf{P}^{3})}\otimes{V}
\rightarrow\mathcal{Q}_{\Lambda(\mathbf{P}^{3})}
$.
Then we have the natural exact sequence
$$
0\longrightarrow
\varPsi^{*}\mathcal{S}_{\Lambda(\mathbf{P}^{3})}
\overset{\sigma}{\longrightarrow}
\varPhi^{*}\mathcal{S}_{\mathbf{P}^{3}}
\overset{\lambda}{\longrightarrow}
\varPsi^{*}\mathcal{Q}_{\Lambda(\mathbf{P}^{3})}
\overset{\tau}{\longrightarrow}
\varPhi^{*}\mathcal{Q}_{\mathbf{P}^{3}}
\longrightarrow0
$$
of vector bundles on $\Gamma(\mathbf{P}^{3})$,
and we have the exact sequence
\begin{multline*}
0\longrightarrow
\mathcal{O}_{\Gamma(\mathbf{P}^{3})}
\overset{\lambda}{\longrightarrow}
\varPhi^{*}\mathcal{S}_{\mathbf{P}^{3}}^{\vee}\otimes
\varPsi^{*}\mathcal{Q}_{\Lambda(\mathbf{P}^{3})}
\overset{\tau\times{\sigma^{\vee}}}{\longrightarrow}\\
\varPhi^{*}(\mathcal{S}_{\mathbf{P}^{3}}^{\vee}\otimes
\mathcal{Q}_{\mathbf{P}^{3}})\oplus
\varPsi^{*}(\mathcal{S}_{\Lambda(\mathbf{P}^{3})}^{\vee}\otimes
\mathcal{Q}_{\Lambda(\mathbf{P}^{3})})
\overset{\sigma^{\vee}\oplus(-\tau)}{\longrightarrow}
\varPsi^{*}\mathcal{S}_{\Lambda(\mathbf{P}^{3})}^{\vee}\otimes
\varPhi^{*}\mathcal{Q}_{\mathbf{P}^{3}}
\longrightarrow0.
\end{multline*}
Since the homomorphism
$$
T_{\mathbf{P}^{3}\times\Lambda(\mathbf{P}^{3})}
\vert_{\Gamma(\mathbf{P}^{3})}
\simeq
\varPhi^{*}(\mathcal{S}_{\mathbf{P}^{3}}^{\vee}\otimes
\mathcal{Q}_{\mathbf{P}^{3}})\oplus
\varPsi^{*}(\mathcal{S}_{\Lambda(\mathbf{P}^{3})}^{\vee}\otimes
\mathcal{Q}_{\Lambda(\mathbf{P}^{3})})
\overset{\sigma^{\vee}\oplus(-\tau)}{\longrightarrow}
\varPsi^{*}\mathcal{S}_{\Lambda(\mathbf{P}^{3})}^{\vee}\otimes
\varPhi^{*}\mathcal{Q}_{\mathbf{P}^{3}}
$$
is identified with the natural homomorphism to the normal bundle
$
T_{\mathbf{P}^{3}\times\Lambda(\mathbf{P}^{3})}
\vert_{\Gamma(\mathbf{P}^{3})}
\rightarrow
\mathcal{N}_{\Gamma(\mathbf{P}^{3})
/\mathbf{P}^{3}\times\Lambda(\mathbf{P}^{3})},
$
we have the exact sequence
\begin{equation}\label{eu}
 0\longrightarrow
\mathcal{O}_{\Gamma(\mathbf{P}^{3})}
\overset{\lambda}{\longrightarrow}
\varPhi^{*}\mathcal{S}_{\mathbf{P}^{3}}^{\vee}\otimes
\varPsi^{*}\mathcal{Q}_{\Lambda(\mathbf{P}^{3})}
\longrightarrow
T_{\Gamma(\mathbf{P}^{3})}
\longrightarrow0.
\end{equation}
Let $(x_{0},\dots,x_{3})$ be a basis of the vector space
$V=H^{0}(\mathbf{P}^{3},\mathcal{O}_{\mathbf{P}^{3}}(1))$, and let
$(x_{0}^{\vee},\dots,x_{3}^{\vee})$ be the dual basis of
$(x_{0},\dots,x_{3})$.
\begin{lemma}\label{tv}
 $$
 H^{0}(Y_{2},T_{\Gamma(\mathbf{P}^{3})}\vert_{Y_{2}})\simeq
 \frac{V^{\vee}\otimes{V}}
 {\mathbf{C}\cdot\sum_{i=0}^{3}x_{i}^{\vee}\otimes{x_{i}}}
 $$
\end{lemma}
\begin{proof}
 The natural homomorphism
 $\mathcal{O}_{\Gamma(\mathbf{P}^{3})}\otimes{V^{\vee}}\otimes{V}
 \rightarrow
 \varPhi^{*}\mathcal{S}_{\mathbf{P}^{3}}^{\vee}\otimes
 \varPsi^{*}\mathcal{Q}_{\Lambda(\mathbf{P}^{3})}
 $
 induces the isomorphism
 ${V^{\vee}}\otimes{V}\simeq
 H^{0}(\Gamma(\mathbf{P}^{3}),
 \varPhi^{*}\mathcal{S}_{\mathbf{P}^{3}}^{\vee}\otimes
 \varPsi^{*}\mathcal{Q}_{\Lambda(\mathbf{P}^{3})})$.
 By the exact sequence $(\ref{eu})$, we have
 $$
 H^{0}(\Gamma(\mathbf{P}^{3}),T_{\Gamma(\mathbf{P}^{3})})\simeq
 \frac{V^{\vee}\otimes{V}}
 {\mathbf{C}\cdot\sum_{i=0}^{3}x_{i}^{\vee}\otimes{x_{i}}},
 $$
 and we can prove that
 $
 H^{0}(\Gamma(\mathbf{P}^{3}),T_{\Gamma(\mathbf{P}^{3})})\simeq
 H^{0}(Y_{2},T_{\Gamma(\mathbf{P}^{3})}\vert_{Y_{2}})
 $
 by the restriction.
\end{proof}
We define filtration on
$\varPsi^{*}\Sym^{3}\mathcal{Q}_{\Lambda(\mathbf{P}^{3})}$
by
$$
\Fil^{i}
=\Fil^{i}\varPsi^{*}\Sym^{3}\mathcal{Q}_{\Lambda(\mathbf{P}^{3})}
=\mathcal{S}^{\otimes{i}}\otimes
\varPsi^{*}\Sym^{3-i}\mathcal{Q}_{\Lambda(\mathbf{P}^{3})}
\subset
\varPsi^{*}\Sym^{3}\mathcal{Q}_{\Lambda(\mathbf{P}^{3})},
$$
where $\mathcal{S}$ denotes the line bundle defined as the kernel of the
homomorphism
$\varPsi^{*}\mathcal{Q}_{\Lambda(\mathbf{P}^{3})}
\overset{\tau}{\rightarrow}
\varPhi^{*}\mathcal{Q}_{\mathbf{P}^{3}}$.
For $G\in\Sym^{3}V$, we denote by $[G]_{i}$ the
image of $G$ by the natural homomorphism
$$
\Sym^{3}V\longrightarrow
H^{0}(\Lambda(\mathbf{P}^{3}),\frac{\Fil^{0}}{\Fil^{i}}),
$$
and denote by $[G]_{i,Y_{j}}$ its restriction to
$H^{0}(Y_{j},\frac{\Fil^{0}}{\Fil^{i}}\big\vert_{Y_{j}})$.
We remark that $Y_{j}$ is the zeros of the regular section $[F]_{j}$,
and if $i\geq{j}$ then $[F]_{i,Y_{j}}$ is contained in
$H^{0}(Y_{j},\frac{\Fil^{j}}{\Fil^{i}}\big\vert_{Y_{j}})$.
We define the sheaf of $\mathcal{O}_{Y_{2}}$-modules $\mathcal{N}$
as the cokernel of the homomorphism
$$
\mathcal{O}_{Y_{2}}\overset{[F]_{3,Y_{2}}}{\longrightarrow}
\frac{\Fil^{0}}{\Fil^{3}}\big\vert_{Y_{2}}.
$$
We remark that $T_{\Gamma(\mathbf{P}^{3})}\vert_{Y_{2}}$ is a quotient
bundle of $\mathcal{O}_{Y_{2}}\otimes{V^{\vee}\otimes{V}}$, and
$\mathcal{N}$ is a quotient $\mathcal{O}_{Y_{2}}$-module of
$\mathcal{O}_{Y_{2}}\otimes\Sym^{3}V$.
And we can check that the homomorphism
$$
\nu:
{V^{\vee}\otimes{V}}
\longrightarrow
\Sym^{3}V;\
x_{i}^{\vee}\otimes{A}\longmapsto
A\frac{\partial{F}}{\partial{x_{i}}}
$$
induces the homomorphism
$$
T_{\Gamma(\mathbf{P}^{3})}\vert_{Y_{2}}
\longrightarrow
\mathcal{N}.
$$
\begin{lemma}\label{ns}
 There is an exact sequence
 $$
 0\longrightarrow
 T_{Y_{2}}(-\log{Y_{3}})
 \longrightarrow
 T_{\Gamma(\mathbf{P}^{3})}\vert_{Y_{2}}
 \longrightarrow
 \mathcal{N}
 \longrightarrow0.
 $$
\end{lemma}
\begin{proof}
 By the definition of $\mathcal{N}$, we have the exact sequence
 $$
 0\longrightarrow
 \mathcal{N}_{Y_{3}/Y_{2}}
 \longrightarrow
 \mathcal{N}
 \longrightarrow
 \mathcal{N}_{Y_{2}/\Gamma(\mathbf{P}^{3})}
 \longrightarrow0,
 $$
 where we remark that
 $
 \mathcal{N}_{Y_{2}/\Gamma(\mathbf{P}^{3})}
 \simeq
 \frac{\Fil^{0}}{\Fil^{2}}\big\vert_{Y_{2}}
 $,
 and $\mathcal{N}_{Y_{3}/Y_{2}}$ is the cokernel of the homomorphism
 $$
 \mathcal{O}_{Y_{2}}\overset{[F]_{3,Y_{2}}}{\longrightarrow}
 \frac{\Fil^{2}}{\Fil^{3}}\big\vert_{Y_{2}}.
 $$
 Since the kernel of the composition
 $T_{\Gamma(\mathbf{P}^{3})}\vert_{Y_{2}}
 \rightarrow
 \mathcal{N}
 \rightarrow
 \mathcal{N}_{Y_{2}/\Gamma(\mathbf{P}^{3})}$
 is identified with $T_{Y_{2}}$,
 we have the homomorphism
 $T_{Y_{2}}\rightarrow
 \mathcal{N}_{Y_{3}/Y_{2}}$
 and its kernel is identified with $T_{Y_{2}}(-\log{Y_{3}})$.
\end{proof}
\begin{lemma}\label{nv}
 $$
 H^{0}(Y_{2},\mathcal{N})\simeq
 \frac{V\otimes\Sym^{2}V}{\mathbf{C}\cdot
 \sum_{i=0}^{3}x_{i}\otimes\frac{\partial{F}}{\partial{x_{i}}}},
 $$
\end{lemma}
\begin{proof}
 By the homomorphism
 $$
 \varPsi^{*}\Sym^{3}\mathcal{Q}_{\Lambda(\mathbf{P}^{3})}
 \longrightarrow
 \varPhi^{*}\mathcal{Q}_{\mathbf{P}^{2}}\otimes
 \varPsi^{*}\Sym^{2}\mathcal{Q}_{\Lambda(\mathbf{P}^{3})};\
 abc\longmapsto
 \tau(a)\otimes{bc}
 +\tau(b)\otimes{ca}
 +\tau(c)\otimes{ab},
 $$
 we have the isomorphism
 $$
 \frac{\Fil^{0}}{\Fil^{3}}
 \simeq
 \varPhi^{*}\mathcal{Q}_{\mathbf{P}^{2}}\otimes
 \varPsi^{*}\Sym^{2}\mathcal{Q}_{\Lambda(\mathbf{P}^{3})}.
 $$
 The natural homomorphism
 $\mathcal{O}_{\Gamma(\mathbf{P}^{3})}\otimes{V}\otimes\Sym^{2}V
 \rightarrow
 \varPhi^{*}\mathcal{Q}_{\mathbf{P}^{2}}\otimes
 \varPsi^{*}\Sym^{2}\mathcal{Q}_{\Lambda(\mathbf{P}^{3})}$
 induces the isomorphism
 $$
 {V}\otimes\Sym^{2}V\simeq
 H^{0}(\Gamma(\mathbf{P}^{3}),
 \varPhi^{*}\mathcal{Q}_{\mathbf{P}^{2}}\otimes
 \varPsi^{*}\Sym^{2}\mathcal{Q}_{\Lambda(\mathbf{P}^{3})})
 \simeq
 H^{0}(\Gamma(\mathbf{P}^{3}),
 \frac{\Fil^{0}}{\Fil^{3}}),
 $$
 and we can prove that
 $$
 H^{0}(\Gamma(\mathbf{P}^{3}),
 \frac{\Fil^{0}}{\Fil^{3}})
 \simeq
 H^{0}(Y_{2},
 \frac{\Fil^{0}}{\Fil^{3}}\big\vert_{Y_{2}}).
 $$
 By the exact sequence
 $$
 0\longrightarrow
 \mathcal{O}_{Y_{2}}\overset{[F]_{3,Y_{2}}}{\longrightarrow}
 \frac{\Fil^{0}}{\Fil^{3}}\big\vert_{Y_{2}}
 \longrightarrow
 \mathcal{N}
 \longrightarrow0,
 $$
 we have
 $$
 H^{0}(Y_{2},\mathcal{N})\simeq
 \frac{V\otimes\Sym^{2}V}{\mathbf{C}\cdot
 \sum_{i=0}^{3}x_{i}\otimes\frac{\partial{F}}{\partial{x_{i}}}}.
 $$
\end{proof}
\begin{lemma}
 The kernel of the homomorphism
 $$
 H^{1}(Y_{2},T_{Y_{2}}(-\log{Y_{3}}))\longrightarrow
 H^{1}(Y_{2},T_{\Gamma(\mathbf{P}^{3})}\vert_{Y_{2}})
 $$
 is identified with the cokernel of the injective homomorphism
 \begin{equation*}
  \delta\circ\nu:
  V^{\vee}\otimes{V}
   \longrightarrow
   V\otimes\Sym^{2}V;\
   x_{j}^{\vee}\otimes{A}\longmapsto
   A\otimes\frac{\partial{F}}{\partial{x_{j}}}
   +\sum_{i=0}^{3}x_{i}\otimes{A}\frac{\partial^{2}{F}}
   {\partial{x_{i}}\partial{x_{j}}}.
 \end{equation*}
\end{lemma}
\begin{proof}
 By the exact sequence in Lemma~$\ref{ns}$,
 we have the exact sequence
 $$
 H^{0}(Y_{2},T_{\Gamma(\mathbf{P}^{3})}\vert_{Y_{2}})
 \longrightarrow
 H^{0}(Y_{2},\mathcal{N})
 \longrightarrow
 H^{1}(Y_{2},T_{Y_{2}}(-\log{Y_{3}}))\longrightarrow
 H^{1}(Y_{2},T_{\Gamma(\mathbf{P}^{3})}\vert_{Y_{2}}).
 $$
 By Lemma~$\ref{tv}$ and Lemma~$\ref{nv}$, we can check that
 $H^{0}(Y_{2},T_{\Gamma(\mathbf{P}^{3})}\vert_{Y_{2}})
 \rightarrow
 H^{0}(Y_{2},\mathcal{N})$
 is induced by the homomorphism
 $$
 \frac{V^{\vee}\otimes{V}}{\mathbf{C}\cdot
 \sum_{i=0}^{3}x_{i}^{\vee}\otimes{x_{i}}}
 \longrightarrow
 \frac{V\otimes\Sym^{2}V}{\mathbf{C}\cdot
 \sum_{i=0}^{3}x_{i}\otimes\frac{\partial{F}}{\partial{x_{i}}}};\
 x_{j}^{\vee}\otimes{A}\longmapsto
 A\otimes\frac{\partial{F}}{\partial{x_{j}}}
 +\sum_{i=0}^{3}x_{i}\otimes{A}\frac{\partial^{2}{F}}
 {\partial{x_{i}}\partial{x_{j}}}.
 $$
 We remark that the homomorphism $\delta\circ\nu$ is the composition of
 injective homomorphisms
 $
 \nu:
 V^{\vee}\otimes{V}\rightarrow\Sym^{3}V
 $
 and
 $$
 \delta:
 \Sym^{3}V\longrightarrow{V}\otimes\Sym^{2}V;\
 G\longmapsto
 \sum_{i=0}^{3}x_{i}\otimes\frac{\partial{G}}{\partial{x_{i}}}.
 $$
\end{proof}
\begin{remark}
 Since $H^{1}(Y_{2},T_{Y_{2}}(-{Y_{3}}))=0$, the homomorphism
 $$
 H^{1}(Y_{2},T_{Y_{2}}(-\log{Y_{3}}))
 \longrightarrow
 H^{1}(Y_{3},T_{Y_{3}})
 $$
 is injective.
\end{remark}
\begin{lemma}\label{jt}
 The Kodaira-Spencer map
 $
 \kappa:T_{\mathcal{M}}([F])\rightarrow
 H^{1}(Y_{3},T_{Y_{3}})
 $
 is computed by the homomorphism
 $$
 \kappa:T_{\mathcal{M}}([F])\simeq
 \frac{\Sym^{3}V}{\mathbf{C}\cdot{F}}
 \longrightarrow
 \frac{V\otimes\Sym^{2}V}{(\delta\circ\nu)(V^{\vee}\otimes{V})}
 \subset
 H^{1}(Y_{3},T_{Y_{3}});\
 G\longmapsto
 \sum_{i=0}^{3}x_{i}\otimes\frac{\partial{G}}{\partial{x_{i}}},
 $$
 and its image $\kappa(T_{\mathcal{M}}([F]))$ is identified with the
 cokernel of the injective homomorphism
 $
 \nu:
 V^{\vee}\otimes{V}\rightarrow\Sym^{3}V.
 $
\end{lemma}
\begin{proof}
 Let $(F,G_{1},\dots,G_{19})$ be a basis of
 $\Sym^{3}V\simeq{H^{0}(\mathbf{P}^{3},\mathcal{O}_{\mathbf{P}^{3}}(3))}$.
 We have a local coordinate of $\mathcal{M}$ at
 $[F]\in\mathcal{M}\subset\mathbf{P}
 (H^{0}(\mathbf{P}^{3},\mathcal{O}_{\mathbf{P}^{3}}(3))^{\vee})$
 by
 $$
 \begin{matrix}
  \mathbf{C}^{19}&\longrightarrow&
  \mathbf{P}
  (H^{0}(\mathbf{P}^{3},\mathcal{O}_{\mathbf{P}^{3}}(3))^{\vee});\\
  (\mu_{1},\dots,\mu_{19})&\longmapsto&
  F-\sum_{i=1}^{19}\mu_{i}G_{i},
 \end{matrix}
 $$
 and the tangent space of $\mathcal{M}$ at $[F]$ is identified with
 $\frac{\Sym^{3}V}{\mathbf{C}\cdot{F}}$ by
 $$
 T_{\mathcal{M}}([F])\simeq\frac{\Sym^{3}V}{\mathbf{C}\cdot{F}};\
 \frac{\partial}{\partial\mu_{j}}
 \longmapsto
 G_{j}.
 $$
 We have the commutative diagram of exact sequences
 $$
 \begin{matrix}
  0&\longrightarrow&
  T_{Y_{3}}
  &\longrightarrow&
  T_{\mathcal{Y}_{3}}\vert_{Y_{3}}
  &\longrightarrow&
  \mathcal{O}_{Y_{3}}\otimes{T_{\mathcal{M}}([F])}
  &\longrightarrow&0\\
  &&\|&&\downarrow&&\downarrow\tilde{\kappa}&&\\
  0&\longrightarrow&
  T_{Y_{3}}
  &\longrightarrow&
  T_{\Gamma(\mathbf{P}^{3})}\vert_{Y_{3}}
  &\longrightarrow&
  \mathcal{N}_{Y_{3}/\Gamma(\mathbf{P}^{3})}
  &\longrightarrow&0,
 \end{matrix}
 $$
 where
 $T_{\mathcal{Y}_{3}}\vert_{Y_{3}}\rightarrow
 T_{\Gamma(\mathbf{P}^{3})}\vert_{Y_{3}}$
 is induced by the natural projection and $\tilde{\kappa}$
 is defined by
 $$
 \tilde{\kappa}:
 T_{\mathcal{M}}([F])\longrightarrow
 H^{0}(Y_{3},\frac{\Fil^{0}}{\Fil^{3}}\big\vert_{Y_{3}})
 \simeq
 H^{0}(Y_{3},\mathcal{N}_{Y_{3}/\Gamma(\mathbf{P}^{3})});\
 \frac{\partial}{\partial\mu_{j}}
 \longmapsto
 [G_{j}]_{3,Y_{3}}.
 $$
 We can compute the homomorphism $\tilde{\kappa}$ by
 $$
 \begin{matrix}
  \tilde{\kappa}:&
  T_{\mathcal{M}}([F])&\longrightarrow&
  \frac{V\otimes\Sym^{2}V}{\mathbf{C}\cdot
  \sum_{i=0}^{3}x_{i}\otimes\frac{\partial{F}}{\partial{x_{i}}}}&
  \simeq
  H^{0}(Y_{2},\mathcal{N})
  \subset
  H^{0}(Y_{3},\mathcal{N}_{Y_{3}/\Gamma(\mathbf{P}^{3})});\\
  &\frac{\partial}{\partial\mu_{j}}
  &\longmapsto&
  \sum_{i=0}^{3}x_{i}\otimes\frac{\partial{G_{j}}}{\partial{x_{i}}},&
 \end{matrix}
 $$
 and $\tilde{\kappa}$ induces the homomorphism
 $$
 \kappa:T_{\mathcal{M}}([F])\simeq
 \frac{\Sym^{3}V}{\mathbf{C}\cdot{F}}
 \longrightarrow
 \frac{V\otimes\Sym^{2}V}{(\delta\circ\nu)(V^{\vee}\otimes{V})}
 \subset
 H^{1}(Y_{2},T_{Y_{2}}(-\log{Y_{3}}))
 \subset
 H^{1}(Y_{3},T_{Y_{3}}).
 $$
\end{proof}
\begin{lemma}\label{to}
 $
 H^{0}(Y_{2},(\varPhi^{*}\mathcal{Q}_{\mathbf{P}^{3}}\otimes
 {T_{\Gamma(\mathbf{P}^{3})}})
 \vert_{Y_{2}})
 $
 is naturally identified with the cokernel of the injective homomorphism
 $$
 \alpha:
 V\oplus{V}
 \longrightarrow
 V\otimes{V^{\vee}}\otimes{V};\
 A\oplus{B}\longmapsto
 \sum_{i=0}^{3}(x_{i}\otimes{x_{i}^{\vee}}\otimes{A}
 +B\otimes{x_{i}^{\vee}}\otimes{x_{i}})
 $$
\end{lemma}
\begin{proof}
 By the exact sequence
 $$
 0\longrightarrow
 \varPsi^{*}\mathcal{Q}_{\Lambda(\mathbf{P}^{3})}
 \longrightarrow
 \varPhi^{*}\mathcal{Q}_{\mathbf{P}^{3}}
 \otimes{V^{\vee}}\otimes
 \varPsi^{*}\mathcal{Q}_{\Lambda(\mathbf{P}^{3})}
 \longrightarrow
 \varPhi^{*}(\mathcal{Q}_{\mathbf{P}^{3}}
 \otimes\mathcal{S}_{\mathbf{P}^{3}}^{\vee})\otimes
 \varPsi^{*}\mathcal{Q}_{\Lambda(\mathbf{P}^{3})}
 \longrightarrow0,
 $$
 $H^{0}(Y_{2},(\varPhi^{*}(\mathcal{Q}_{\mathbf{P}^{3}}
 \otimes\mathcal{S}_{\mathbf{P}^{3}}^{\vee})\otimes
 \varPsi^{*}\mathcal{Q}_{\Lambda(\mathbf{P}^{3})})\vert_{Y_{2}})$
 is identified with the cokernel of the injective homomorphism
 $$
 \lambda_{0}:V\longrightarrow
 {V}\otimes{V^{\vee}}\otimes{V};\
 A\longmapsto
 \sum_{i=0}^{3}x_{i}\otimes{x_{i}^{\vee}}\otimes{A}.
 $$
 By the exact sequence $(\ref{eu})$,
 we have the exact sequence
 $$
 0\longrightarrow
 \varPhi^{*}\mathcal{Q}_{\mathbf{P}^{3}}
 \overset{\lambda}{\longrightarrow}
 \varPhi^{*}(\mathcal{Q}_{\mathbf{P}^{3}}
 \otimes\mathcal{S}_{\mathbf{P}^{3}}^{\vee})\otimes
 \varPsi^{*}\mathcal{Q}_{\Lambda(\mathbf{P}^{3})}
 \longrightarrow
 \varPhi^{*}\mathcal{Q}_{\mathbf{P}^{3}}
 \otimes{T_{\Gamma(\mathbf{P}^{3})}}
 \longrightarrow0,
 $$
 and
 $H^{0}(Y_{2},(\varPhi^{*}\mathcal{Q}_{\mathbf{P}^{3}}
 \otimes{T_{\Gamma(\mathbf{P}^{3})}})\vert_{Y_{2}})$
 is identified with the cokernel of the injective homomorphism
 $$
 V\longrightarrow
 \frac{V\otimes{V^{\vee}}\otimes{V}}{\lambda_{0}(V)};\
 B\longmapsto
 \sum_{i=0}^{3}B\otimes{x_{i}^{\vee}}\otimes{x_{i}}.
 $$
\end{proof}
\begin{lemma}\label{no}
 $
 H^{0}(Y_{2},
 (\varPhi^{*}\mathcal{Q}_{\mathbf{P}^{3}})\vert_{Y_{2}}\otimes\mathcal{N})
 $
 is naturally identified with the cokernel of the injective homomorphism
 $$
 \beta:
 V\oplus{V}
 \longrightarrow
 \Sym^{2}V\otimes\Sym^{2}V;\
 A\oplus{B}\longmapsto
 \sum_{i=0}^{3}\bigl(\frac{\partial{F}}{\partial{x_{i}}}\otimes{Ax_{i}}
 +Bx_{i}\otimes\frac{\partial{F}}{\partial{x_{i}}}\bigr)
 $$
\end{lemma}
\begin{proof}
 By the exact sequence
 \begin{multline*}
  0\longrightarrow
  (\mathcal{S}^{\vee}
  \otimes\varPsi^{*}\Sym^{2}\mathcal{Q}_{\Lambda(\mathbf{P}^{3})})\vert_{Y_{1}}
  \overset{[F]_{2,Y_{1}}}{\longrightarrow}
  (\varPhi^{*}\mathcal{Q}_{\mathbf{P}^{3}}^{\otimes{2}}
  \otimes\varPsi^{*}\Sym^{2}\mathcal{Q}_{\Lambda(\mathbf{P}^{3})})
  \vert_{Y_{1}}\\
  \longrightarrow
  (\varPhi^{*}\mathcal{Q}_{\mathbf{P}^{3}}^{\otimes{2}}
  \otimes\varPsi^{*}\Sym^{2}\mathcal{Q}_{\Lambda(\mathbf{P}^{3})})\vert_{Y_{2}}
  \rightarrow0,
 \end{multline*}
 $H^{0}(Y_{2},(\varPhi^{*}\mathcal{Q}_{\mathbf{P}^{3}}
 \otimes\frac{\Fil^{0}}{\Fil^{3}})\vert_{Y_{2}})
 \simeq
 H^{0}(Y_{2},(\varPhi^{*}\mathcal{Q}_{\mathbf{P}^{3}}^{\otimes{2}}
 \otimes\varPsi^{*}\Sym^{2}\mathcal{Q}_{\Lambda(\mathbf{P}^{3})})
 \vert_{Y_{2}})$
 is identified with the cokernel of the injective homomorphism
 $$
 [F]_{2}:
 V\longrightarrow\Sym^{2}V\otimes\Sym^{2}V;\
 A\longmapsto
 \sum_{i=0}^{3}\frac{\partial{F}}{\partial{x_{i}}}\otimes{Ax_{i}}.
 $$
 By the exact sequence
 $$
 0\longrightarrow
 (\varPhi^{*}\mathcal{Q}_{\mathbf{P}^{3}})\vert_{Y_{2}}
 \overset{[F]_{3,Y_{2}}}{\longrightarrow}
 \bigl(\varPhi^{*}\mathcal{Q}_{\mathbf{P}^{3}}
 \otimes\frac{\Fil^{0}}{\Fil^{3}}\bigr)\big\vert_{Y_{2}}
 \longrightarrow
 (\varPhi^{*}\mathcal{Q}_{\mathbf{P}^{3}})\vert_{Y_{2}}
 \otimes\mathcal{N}
 \longrightarrow0,
 $$
 $H^{0}(Y_{2},(\varPhi^{*}\mathcal{Q}_{\mathbf{P}^{3}})\vert_{Y_{2}}
 \otimes\mathcal{N})$
 is identified with the cokernel of the injective homomorphism
 $$
 V\longrightarrow\frac{\Sym^{2}V\otimes\Sym^{2}V}{[F]_{2}(V)};\
 B\longmapsto
 \sum_{i=0}^{3}B{x_{i}}\otimes\frac{\partial{F}}{\partial{x_{i}}}.
 $$
\end{proof}
\begin{lemma}\label{jm}
 $H^{1}(Y_{2},\Omega_{Y_{2}}^{2}(\log{Y_{3}}))$
 is naturally identified with the cokernel of the injective homomorphism
 \begin{equation*}
  \nu_{1}:
   V\otimes{V^{\vee}}\otimes{V}
   \longrightarrow
   \Sym^{2}V\otimes\Sym^{2}V;\
   A\otimes{x_{j}^{\vee}}\otimes{B}\longmapsto
   AB\otimes\frac{\partial{F}}{\partial{x_{j}}}
   +\sum_{i=0}^{3}Ax_{i}\otimes{B}\frac{\partial^{2}{F}}
   {\partial{x_{i}}\partial{x_{j}}}.
 \end{equation*}
\end{lemma}
\begin{proof}
 Since
 $$
 \Omega_{Y_{2}}^{2}(\log{Y_{3}})\simeq
 \Omega_{Y_{2}}^{3}({Y_{3}})\otimes{T_{Y_{2}}(-\log{Y_{3}})}\simeq
 (\varPhi^{*}\mathcal{Q}_{\mathbf{P}^{3}})\vert_{Y_{2}}
 \otimes{T_{Y_{2}}(-\log{Y_{3}})},
 $$
 we have the exact sequence
 $$
 0\longrightarrow
 \Omega_{Y_{2}}^{2}(\log{Y_{3}})
 \longrightarrow
 (\varPhi^{*}\mathcal{Q}_{\mathbf{P}^{3}}
 \otimes{T_{\Gamma(\mathbf{P}^{3})}})\vert_{Y_{2}}
 \longrightarrow
 (\varPhi^{*}\mathcal{Q}_{\mathbf{P}^{3}})\vert_{Y_{2}}\otimes\mathcal{N}
 \longrightarrow0
 $$
 by Lemma~$\ref{ns}$, and we can check that
 $H^{1}(Y_{2},(\varPhi^{*}\mathcal{Q}_{\mathbf{P}^{3}}
 \otimes{T_{\Gamma(\mathbf{P}^{3})}})
 \vert_{Y_{2}})=0$.
 By Lemma~$\ref{to}$ and Lemma~$\ref{no}$,
 $H^{1}(Y_{2},\Omega_{Y_{2}}^{2}(\log{Y_{3}}))$ is identified with the
 cokernel of the homomorphism
 $$
 \frac{V\otimes{V^{\vee}}\otimes{V}}{\alpha(V\oplus{V})}
 \longrightarrow
 \frac{\Sym^{2}V\otimes\Sym^{2}V}{\beta(V\oplus{V})};\
 A\otimes{x_{j}^{\vee}}\otimes{B}\longmapsto
 AB\otimes\frac{\partial{F}}{\partial{x_{j}}}
 +\sum_{i=0}^{3}Ax_{i}\otimes{B}\frac{\partial^{2}{F}}
 {\partial{x_{i}}\partial{x_{j}}},
 $$
 and it is injective because
 $H^{0}(Y_{2},\Omega_{Y_{2}}^{2}(\log{Y_{3}}))=0$.
 Since the homomorphism $\nu_{1}$ induces an isomorphism
 $\alpha(V\oplus{V})\simeq\beta(V\oplus{V})$,
 the homomorphism $\nu_{1}$ is injective.
\end{proof}
\begin{proof}[Proof of Proposition~$\ref{iv}$]
 By Lemma~$\ref{jt}$ and Lemma~$\ref{jm}$,
 we have a commutative diagram of exact sequences
 $$
 \begin{matrix}
  0&&0\\
  \downarrow&&\downarrow\\
  V\otimes{V^{\vee}}\otimes{V}&=&V\otimes{V^{\vee}}\otimes{V}\\
  {1\otimes\nu}\downarrow\hspace*{32pt}
  &&\hspace*{15pt}\downarrow\nu_{1}\\
  V\otimes\Sym^{3}V&\overset{\delta_{1}}{\longrightarrow}&
  \Sym^{2}V\otimes\Sym^{2}V\\
  \downarrow&&\downarrow\\
  H^{0}(Y_{3},\Omega_{Y_{3}}^{2})\otimes
  \kappa(T_{\mathcal{M}}([F]))
  &\longrightarrow&
  H^{1}(Y_{3},\Omega_{Y_{3}}^{1})\\
  \downarrow&&\\
  0,&&\\
 \end{matrix}
 $$
 where we remark that
 $$
 V\simeq
 H^{0}(Y_{2},(\varPhi^{*}\mathcal{Q}_{\mathbf{P}^{3}})\vert_{Y_{2}})
 \simeq
 H^{0}(Y_{2},\Omega_{Y_{2}}^{3}(Y_{3}))
 \simeq
 H^{0}(Y_{3},\Omega_{Y_{3}}^{2}),
 $$
 $$
 \frac{\Sym^{2}{V}\otimes\Sym^{2}{V}}
 {\nu_{1}(V\otimes{V^{\vee}}\otimes{V})}
 \simeq
 H^{1}(Y_{2},\Omega_{Y_{2}}^{2}(\log{Y_{3}}))
 \subset
 H^{1}(Y_{3},\Omega_{Y_{3}}^{1})
 $$
 and the homomorphism $\delta_{1}$ is defined by
 $$
 \delta_{1}:
 V\otimes\Sym^{3}V\longrightarrow\Sym^{2}V\otimes\Sym^{2}V;\
 A\otimes{B}\longmapsto
 \sum_{i=0}^{3}Ax_{i}\otimes\frac{\partial{B}}{\partial{x_{i}}}.
 $$
 Since $\delta_{1}$ is injective, the homomorphism
 $H^{0}(Y_{3},\Omega_{Y_{3}}^{2})\otimes
 \kappa(T_{\mathcal{M}}([F]))
 \rightarrow
 H^{1}(Y_{3},\Omega_{Y_{3}}^{1})$
 is also injective, hence the dimension of its image is equal to $16$.
\end{proof}

\bigskip
\begin{flushleft}
 \textsc{Graduate School of Science\\
 Osaka University\\
 Toyonaka, Osaka, 560-0043\\
 Japan}\\
 \textit{E-mail address}:
 {\ttfamily atsushi@math.sci.osaka-u.ac.jp}
\end{flushleft}
\end{document}